\documentclass[12pt]{amsart}
\usepackage[lmargin=1in,rmargin=1in,
bmargin=1in, tmargin=1in]{geometry}

\title[Multiplicity-free representations over quasi-symmetric Siegel domains]{Multiplicity-free representations and\\
coisotropic actions of certain\\
nilpotent Lie groups over\\
quasi-symmetric Siegel domains}
\author{Koichi Arashi}
\address{K. Arashi: Department of Mathematics, Tokyo Gakugei University, Nukuikita 4-1-1, Koganei,Tokyo 184-8501, Japan}
\email{arashi@u-gakugei.ac.jp}
\usepackage{amsmath}
\usepackage{graphics}
\usepackage{comment}
\usepackage{amsrefs}
\usepackage{stmaryrd}
\usepackage{amssymb}
\usepackage{amsthm}
\usepackage{mathrsfs}
\usepackage{tikz}
\usepackage{framed}
\usepackage{mathrsfs}
\usepackage{color}
\usepackage[utf8]{inputenc}
\usepackage[english]{babel}
\usepackage{hyperref}
\usepackage{enumitem}
\usepackage{csquotes}

\hypersetup{
    colorlinks=true,
    linkcolor=blue,
    filecolor=magenta,      
    urlcolor=cyan,
}
\numberwithin{equation}{section}
\theoremstyle{definition}

\newtheorem{definition}{Definition}[section]
\newtheorem{remark}{Remark}[section]
\theoremstyle{plain}
\newtheorem{lemma}{Lemma}[section]
\newtheorem{theorem}{Theorem}[section]
\newtheorem{proposition}{Proposition}[section]
\newtheorem{corollary}{Corollary}[section]

\newcommand{\boxx}[2]{#1\,\Box\,#2}
\DeclareMathOperator{\re}{Re}
\DeclareMathOperator{\im}{Im}
\DeclareMathOperator{\cl}{cl}
\DeclareMathOperator{\tr}{tr}
\DeclareMathOperator{\End}{End}
\DeclareMathOperator{\ext}{ext}
\DeclareMathOperator{\Ad}{Ad}
\DeclareMathOperator{\Id}{Id}

\DeclareMathOperator{\Sym}{Sym}
\DeclareMathOperator{\rank}{rank}

\definecolor{shadecolor}{gray}{0.80}

\begin{document}

\begin{abstract}
We study multiplicity-free representations of Lie groups over a quasi-symmetric Siegel domain, with a focus on certain two-step nilpotent Lie groups.
We provide necessary and sufficient conditions for the multiplicity-freeness property to hold.
Specifically, we establish the equivalence between the disjointness of irreducible unitary representations realized over the domain, the multiplicity-freeness of the unitary representation on the Bergman space, and the coisotropicity of the group action.
\bigskip

\end{abstract}
\maketitle
\tableofcontents
\section{Introduction}
Addressing intricate problems in the representation theory of Lie groups from a foundational perspective, particularly those involving irreducible representations, is a basic approach and is particularly effective when studying unitary representations on spaces of holomorphic functions.

In this paper, our focus is on the quasi-symmetric Siegel domain, a generalization of a Hermitian symmetric space of noncompact type. 
Let $\Omega$ be a symmetric cone contained in a finite-dimensional real inner product space $(U,\langle\cdot,\cdot\rangle_U)$.
Fix a reference point $e \in \Omega$, assuming it satisfies a certain compatibility condition with $\langle \cdot, \cdot \rangle_U$, and consider the Jordan algebra structure on $U$.
Let $(V, h)$ be a finite-dimensional complex inner product space, and consider a unital Jordan algebra homomorphism $x \mapsto 2R_x$ from $U$ into $\mathcal{H}(V, h)$, the set of all self-adjoint operators on $V$.
From this, we define a Hermitian map $Q: V \times V \rightarrow U_\mathbb{C}$ (see Sections \ref{subsect:selfdual} and \ref{sect:quasisymmetric} for details).
This gives rise to a quasi-symmetric Siegel domain
\begin{align*}
\mathcal{S}(\Omega, Q) = \{ (z, v) \in U_\mathbb{C} \times V \mid \im z - Q(v, v) \in \Omega \}.
\end{align*}
Here, we note that a more general notion of the Siegel domain can be defined solely by a regular cone $\Omega_0 \subset U$ and an $\Omega_0$-positive Hermitian map $Q_0: V \times V \rightarrow U_\mathbb{C}$, denoted by $\mathcal{S}(\Omega_0, Q_0)$.
Moreover, a quasi-symmetric Siegel domain is originally defined as a Siegel domain that does not necessarily satisfy one of the conditions characterizing symmetric domains (see \cite{satake_algebraic_2014} for details).
It is known \cite{gindikin_analysis_1964} that the Bergman kernel $K$ of the domain, i.e., the reproducing kernel of the space $L_a^2(\mathcal{S}(\Omega, Q))$ of all $L^2$ holomorphic functions on $\mathcal{S}(\Omega,Q)$ with respect to the measure defined by certain bases of $U$ and $V$ (see Sect. \ref{sect:L2} for details) can be expressed as follows :
\begin{align}\label{eq:intabsintegralexpression}
K(z,v,z',v') = \int_\Omega e^{i\langle x, z - \overline{z'} - 2iQ(v,v') \rangle_U} \, dm(x),
\end{align}
where $m$ is a measure on $\Omega$ equivalent to the natural complete measure and $\langle\cdot,\cdot \rangle_U$ is extended to a complex bilinear form.
In this formula, the generalized Heisenberg group $G^V = U \rtimes V$ which consists of all affine transformations of $\mathcal{S}(\Omega, Q)$ that preserve the values $\im z - Q(v,v) \in \Omega$ plays a significant role.
To clarify, for a complex manifold $\mathcal{D}$ let $\mathcal{O}(\mathcal{D})$ denote the space of all holomorphic functions on $\mathcal{D}$, and for a group $G_0$ of holomorphic automorphisms of $\mathcal{D}$, let $\Gamma_{G_0}(\mathcal{D})$ denote the convex cone consisting of all $G_0$-invariant reproducing kernels.
Then according to \cite{schwartz_sousespaces_1964}, there exists a bijective correspondence between the set $\Gamma_{G_0}(\mathcal{D})$ and the set of all $G_0$-invariant Hilbert subspaces of $\mathcal{O}(\mathcal{D})$, which give rise to unitary representations of $G_0$.
Returning to our setting, the integrands $e^{i\langle x, z - \overline{z'} - 2iQ(v,v') \rangle_U}\,( x \in \Omega)$ in \eqref{eq:intabsintegralexpression} lie in $\Gamma_{G^V}(\mathcal{S}(\Omega, Q))$, and they are non-overlapping with respect to the unitary dual $\widehat{G^V}$.
A similar phenomenon occurs in other function spaces beyond $L_a^2(\mathcal{S}(\Omega, Q))$ (see \cites{nussbaum_hausdorffbernsteinwidder_1955, vergne_analytic_1976, ishi_representations_1999}), leading to the question of whether the condition \begin{align*}
\ext(\Gamma_{G^V}(\mathcal{S}(\Omega,Q)))/\mathbb{R}^\times \hookrightarrow \widehat{G^V}
\end{align*}
holds.
Here, $\ext$ stands for the extremal points, and the condition represents a certain form of multiplicity-freeness property, and we can also replace $G^V$ by other groups.

In relation to these multiplicity-free conditions, connections to the geometrical properties of group orbits of the base space have been investigated.
For a connected compact subgroup $G_0$ of the holomorphic automorphism group of $\mathcal{S}(\Omega, Q)$, the condition $\ext(\Gamma_{G_0}(\mathcal{S}(\Omega, Q)))/\mathbb{R}^\times \hookrightarrow \widehat{G_0}$ holds if and only if the action of $G_0$ on $\mathcal{S}(\Omega, Q)$ is coisotropic with respect to the symplectic form induced by the Bergman metric of $\mathcal{S}(\Omega,Q)$ \cite{huckleberry_multiplicityfree_1990}.
This fact offers an additional perspective on the multiplicity-freeness property explored in this paper.
In addition to this, we note that coisotropic actions of groups on symplectic manifolds exhibit significant connections to the representation theory (see, e.g., \cites{guillemin_multiplicityfree_1984, vinberg_commutative_2001, knop_classification_2006, geatti_polar_2017, kitagawa_uniformly_}).
Although the groups considered in this paper are nilpotent, we also explore the coisotropicity of group orbits, which requires a different approach in the proof.
We also note that the multiplicity-free theorems in \cites{faraut_invariant_1999, kobayashi_propagation_2013a}, which apply to not necessarily compact Lie groups, such as in the settings of symmetric spaces (e.g., \cites{kobayashi_visible_2007,sasaki_characterization_2010}), spherical varieties (e.g.,\cites{kobayashi_generalized_2007a,tanaka_classification_2012,tanaka_visible_2022}), and Heisenberg homogeneous spaces \cite{baklouti_visible_2021}, were further explored in \cites{arashi_visible_2022a, arashi_multiplicityfree_a} in the context of Siegel domains.

In this paper, toward understanding the necessary and sufficient conditions for the multiplicity-freeness property of representations of nilpotent Lie groups over K\"ahler manifolds, we focus on the subgroups $G^W = U \rtimes W$ of $G^V$, where $W\subset V$ is a real subspace.
Note that, in general, the property under consideration can be derived from the corresponding properties of smaller groups.
Let $j$ be the complex structure on $V$ and $S(=S(W)):=jW^{\perp,\,\re h}$ the orthogonal complement of $jW$ with respect to $\re h$.
The main theorem of this paper is as follows. 
While the implication $\ref{condi:mainmf}\Rightarrow\ref{condi:mainl2}$ of the theorem follows from \cite[Theorem 2]{faraut_invariant_1999}, the remaining parts constitute the primary contribution of this paper.
\begin{theorem}[see Corollary \ref{cor:112} and Theorems \ref{th:220}, \ref{th:5.11}, \ref{th:5.14}]\label{th:main}
For a real subspace $W\subset V$, the following conditions are equivalent:
\begin{enumerate}[label=\textup{(\roman*)}]
\item
$\ext(\Gamma_{G^W}(\mathcal{S}(\Omega,Q)))/\mathbb{R}^\times\hookrightarrow\widehat{G^W}$;\label{condi:mainmf}
\item
The natural unitary representation of $G^W$ on $L_a^2(\mathcal{S}(\Omega,Q))$ is multiplicity-free;\label{condi:mainl2}
\item
$\im Q(S,S)=\{0\}$;\label{condi:mainreal}
\item
Each $G^W$-orbit of $\mathcal{S}(\Omega, Q)$ is a coisotropic submanifold with respect to the symplectic form induced by the Bergman metric.\label{condi:maincoiso}
\end{enumerate}
\end{theorem}
Our proof of \ref{condi:mainreal} $\Rightarrow$ \ref{condi:mainmf} of Theorem \ref{th:main} is founded on determining the set $\ext(\Gamma_{G^W}(\mathcal{S}(\Omega,Q)))$ and constructing intertwining operators between the holomorphic induced representations in \cite{auslander_polarization_1971} and unitary representations of $G^W$ realized in $\mathcal{O}(\mathcal{S}(\Omega,Q))$.
For $x \in U$, let $g_x \in \Sym^2((V_\mathbb{R})^*)$ be given by $g_x(v, v') = \langle x, \re Q(v, v') \rangle_U$ for $v, v' \in V$.
In our previous work \cite{arashi_multiplicityfree_a}, we primarily focused on two cases: when $W = V$ and when $W$ is a real form of $V$.
In contrast to the latter case, for $x \in \Omega$, the condition condition that $W$ is coisotropic with respect to the symplectic form $\omega_x(\cdot,\cdot)=g_x(\cdot,j\cdot)$ does not necessarily imply that $W$ is isotropic with respect to the same symplectic form.
This introduces new complexities in our analysis.
We overcome this difficulty by developing useful tools based on the spectral theorem of Jordan algebras, and by employing the pseudo-inverse map for a Siegel domain (see, e.g., \cites{dorfmeister_homogeneous_1982, nomura_geometric_2003,kai_characterization_2005}), a generalization of Vinberg's $\ast$-map \cite{vinberg_theory_1963}, particularly in the proof of the equivalence \ref{condi:mainreal} $\Leftrightarrow$ \ref{condi:maincoiso} of Theorem \ref{th:main}.

Using the condition \ref{condi:mainreal} along with the necessary and sufficient conditions in \cite[Theorem 1.2]{arashi_multiplicityfree_a} concerning the visible action \cite{kobayashi_multiplicityfree_2005}, we derive the following corollary.
\begin{corollary}\label{cor:main}
The following condition can be added to Theorem \ref{th:main}:
\begin{enumerate}[label=\textup{(\roman*)}, start=5]
\item
$S\cap jS=\{0\}$ and the action of $G^S$ on $\mathcal{S}(\Omega,Q|_{(S+jS)\times (S+jS)})$ is strongly visible with respect to an involutive anti-holomorphic diffeomorphism.
\end{enumerate}
\end{corollary}
In addition, we note that under any condition in Theorem \ref{th:main}, the action of $G^W$ on $\mathcal{S}(\Omega,Q)$ is visible.

We now describe an admissible parametrization of $\ext(\Gamma_{G^W}(\mathcal{S}(\Omega, Q)))$, a certain family of extremal $G^W$-invariant reproducing kernels, allowing each $K \in \Gamma_{G^W}(\mathcal{S}(\Omega, Q))$ to be expressed as an integral of them (see Sect. \ref{subsect:invariant} for details).
Instead of presenting the general result here, we offer an example with a concrete description.
Fixing a Jordan frame $e_1,e_2,\cdots, e_r$ of $U$, for $x\in U$, let us consider the Peirce decomposition $x=x_1+x_{1/2}+x_0$ with $x_\lambda\in U(e_1,\lambda)=\{x\in U\mid T_{e_1}x=\lambda x\}\,(\lambda=1,1/2,0)$.
Here, the left multiplication of an element $u\in U_0$ in a Jordan algebra $U_0$ is denoted by $T_{u}$.
We assume that the Euclidean Jordan algebra $U$ is simple, $R_x\neq 0$ for all $x\neq 0$, and $\langle e_1,e_1\rangle_U=1$.
Noting that $U(e_1,0)$ is a subalgebra of $U$,  for $0\leq k\leq r-1$, let
\begin{align}\label{eq:lambdak}
\Lambda_k=\left\{{x}\in U\Bigm| \begin{array}{c}T_{x_0}\in\End(U(e_1,0))\text{ is positive semi-definite},\, \rank x_0=k,\\\text{ there exists }y\in U(e_1,1/2)\text{ such that }x_{1/2}=2T_yx_0\end{array}\right\}.
\end{align}
Combining our classification of extremal $G^W$-invariant reproducing kernels and the general theory about the admissible parametrization \cite{faraut_invariant_1999}, we have the following theorem.
In this theorem, for a real vector space $W_0$, we denote the complex conjugate of $v\in (W_0)_\mathbb{C}$ with respect to $W_0$ by $\overline{v}$.
\begin{theorem}[see Corollary \ref{cor:115} and Sect. \ref{subsect:specificform}]\label{th:main2}
For a suitable choice of a real form $S$ of $R_{e_1}V$ and a complex subspace $P \subset V$ such that $V = R_{e_1}V \oplus P$, any condition in Theorem \ref{th:main} holds for $W = P \oplus S$.
Moreover, for any $G^W$-invariant Hilbert subspace $\mathcal{H}$ of $\mathcal{O}({\mathcal{S}(\Omega,Q)})$, there exists unique Radon measures $m_{k}$ on $\Lambda_k\times S^*\,(k=0,1,\cdots,r-1)$ such that the reproducing kernel $K^\mathcal{H}$ of $\mathcal{H}$ is expressed as
\begin{align*}
K^\mathcal{H}(z,v,z',v')=\sum_{k=0}^{r-1}\int_{\Lambda_k\times S^*}L^{{x},\chi}(z,v,z',v')\,dm_k({x},\chi).
\end{align*}
Here, for $(x,\chi)\in \Lambda_k\times S^*$, letting $y\in U(e_1,1/2)$ given as in \eqref{eq:lambdak} and extending $\chi$ to a linear form on $R_{e_1}V$ by the complex linearity, the function $L^{x,\chi}\in\ext(\Gamma_{G^W}(\mathcal{S}(\Omega,Q)))$ is defined by
\begin{align*}
L^{{x},\chi}(z,q+s,z',q'+s')&=e^{i\langle {x},z-\overline{z'}-2iQ(q,q')\rangle_U}e^{-i\langle\chi,s-\overline{s'}\rangle}
\\&\cdot \exp({\langle x_1-2T_{e_1}{(T_y)}^2x_0,e_1\rangle_U(h(s,\overline{s})}+h(\overline{s'},s')))
\end{align*}
with $q,q'\in P$, $s,s'\in R_{e_1}V$.
\end{theorem}

The conditions in \eqref{eq:lambdak} are derived from the positivity of the Fubini-Study metric on the infinite-dimensional projective space.
This idea originates from the study of coherent state representations \cites{lisiecki_kaehler_1990,lisiecki_classification_1991, lisiecki_coherent_1995}, which extend the concept of highest weight unitary representations (see further developments in \cite{neeb_holomorphy_1999}).

We now outline the organization of this paper.
In Sect. \ref{sect:preliminaries}, we review preliminary results relevant to our study.
Sect. \ref{subsect:invariant} introduces the notions of invariant Hilbert subspaces and the multiplicity-freeness property in the context of complex manifolds admitting group actions, providing the framework and motivation for our main theorem.
In Sect. \ref{subsect:orbit}, we briefly recall general aspects of the orbit method, which serves as a fundamental framework for analyzing unitary representations of nilpotent Lie groups.
Sect. \ref{subsect:selfdual} discusses the Euclidean Jordan algebra defined by a self-dual homogeneous cone and presents a spectral theorem.
In Sect. \ref{sect:quasisymmetric}, we introduce the notion of quasi-symmetric Siegel domain and provide useful lemmas, along with an important proposition that forms the basis for the subsequent analysis, particularly in the next section.
Sect. \ref{subsect:extremal} determines all extremal $G^W$-invariant reproducing kernels and describes an admissible parametrization of $\ext(\Gamma_{G^W}(\mathcal{S}(\Omega, Q)))$.
In Sect. \ref{subsect:intertwiningoperator}, we prove \ref{condi:mainreal} $\Rightarrow$ \ref{condi:mainmf} of Theorem \ref{th:main}.
In Sect. \ref{subsect:specificform} we focus on the special case presented in Theorem \ref{th:main2}.
In Sect. \ref{sect:L2}, we prove \ref{condi:mainl2} $\Rightarrow$ \ref{condi:mainreal} of Theorem \ref{th:main}.
In Sect. \ref{sect:coisotropic}, we show the equivalence \ref{condi:mainreal} $\Leftrightarrow$ \ref{condi:maincoiso} of Theorem \ref{th:main}.

\section{Preliminaries}\label{sect:preliminaries}
In this section, we recall the key definitions, notions, and results that are essential for the development of our study.

Throughout this paper, for a Lie group, we denote its Lie algebra by the corresponding Fraktur small letter.
For a vector space $V$ over $\mathbb{C}$ and its real form $W$, we denote the complex conjugate of $v\in V$ with respect to $W$ by $\overline{v}^W$.
The vector space over $\mathbb{R}$ obtained from $V$ by restricting the scalars to $\mathbb{R}$ will be denoted by $V_\mathbb{R}$.
For a vector space or a Lie algebra $W$ over $\mathbb{R}$, we denote by $W_\mathbb{C}$ its complexification $W\otimes_\mathbb{R}\mathbb{C}$.
The natural complex conjugate of $v\in W_\mathbb{C}$ will be denoted by $\overline{v}$.
We may extend $\xi\in W^*$ to a linear form on $W_\mathbb{C}$ by the complex linearity without making any comment.

\subsection{Invariant Hilbert subspaces and the multiplicity-freeness property}\label{subsect:invariant}
In this subsection, we review the theoretical framework studied in \cite{faraut_invariant_1999} that supports our study of representations of Lie groups on spaces of holomorphic functions, with particular attention to the multiplicity-freeness property.

Let $G$ be a Lie group, and $\pi$ a unitary representation on a separable Hilbert space $\mathcal{H}$.
Put
\begin{align*}
\End_G(\mathcal{H}):=\{A\in B(\mathcal{H})\mid A\pi(g)=\pi(g)A\text{ for all }g\in G\}.
\end{align*}
\begin{definition}
We say $\pi$ is {\it multiplicity-free} if the ring $\End_G(\mathcal{H})$ is commutative.
\end{definition}
For a complex domain $\mathcal{D}$, let $\mathcal{O}(\mathcal{D})$ be the space of holomorphic functions on $\mathcal{D}$, which we regard as a topological vector space by means of the compact-open topology.
A smooth action $G\times\mathcal{D}\ni (g,z)\mapsto g\cdot z\in\mathcal{D}$ of a Lie group $G$ by holomorphic automorphisms defines a continuous representation $(\pi_0,\mathcal{O}(\mathcal{D}))$ given by
\begin{align*}
\pi_0(g)f(z):=f(g^{-1}\cdot z)\quad(g\in G, f\in\mathcal{O}(\mathcal{D}), z\in\mathcal{D}).
\end{align*}
Next, we introduce the notions of the multiplicity-freeness of $\pi_0$ as follows.
\begin{definition}[{\cite[\S1]{schwartz_sousespaces_1964}}, {\cite[Definition 2.1]{kobayashi_propagation_2013a}}]\label{def:4.1}
\begin{enumerate}[label*=\textup{(\arabic*)}]
\item
We say $\pi$ is {\it realized in} $\mathcal{O}(\mathcal{D})$ if there exists an injective continuous $G$-intertwining operator $\Phi$ between $\pi$ and $\pi_0$.
In this case, we call the image $\Phi(\mathcal{H})$ with the induced inner product (or simply $\mathcal{H}$) a $G$-{\it invariant Hilbert subspace} of $\mathcal{O}(\mathcal{D})$.
Moreover, we use the terminology `irreducible $G$-invariant Hilbert subspace' when $\pi$ is irreducible.
\item
We say $\pi_0$ is {\it multiplicity-free} if any two irreducible $G$-invariant Hilbert subspaces of $\mathcal{O}(\mathcal{D})$ either coincide as linear spaces and have proportional inner products, or they yield inequivalent representations of $G$.
\end{enumerate}
\end{definition}
Let $\Gamma(\mathcal{D})$ be the convex cone of functions $K(z,z')$ on $\mathcal{D}\times\mathcal{D}$ holomorphic in $z$, anti-holomorphic in $z'$, and Hermitian of positive type, and $\Gamma_G(\mathcal{D})\subset\Gamma(\mathcal{D})$ be the convex cone consisting of all $G$-invariant functions.
Recall that an element $K$ of a convex cone $\Gamma$ is called {\it extremal} if any decomposition $K=K_1+K_2\,(K_1,K_2\in\Gamma)$ yields
\begin{align*}
K=\lambda_1K_1=\lambda_2K_2\quad(\lambda_1,\lambda_2\geq 0).
\end{align*}
Let $\ext (\Gamma_G(\mathcal{D}))\subset\Gamma_G(\mathcal{D})$ be the subset consisting of all extremal elements.
The following result is fundamental to our study.
\begin{theorem}[{\cite[\S 8]{schwartz_sousespaces_1964}},{\cite[Proposition 1]{faraut_invariant_1999}}]\label{th:243}
The elements of $\Gamma_G(\mathcal{D})$ (resp. $\ext (\Gamma_G(\mathcal{D}))$) stand in one-one correspondence with $G$-invariant (resp. irreducible $G$-invariant) Hilbert subspaces of $\mathcal{O}(\mathcal{D})$.
\end{theorem}
\begin{remark}
For a $G$-invariant Hilbert subspace $\mathcal{H}$ of $\mathcal{O}(\mathcal{D})$, the corresponding function $K^\mathcal{H}$, known as the {\it reproducing kernel} of $\mathcal{H}$, can be defined by
\begin{align*}
(f,K_z^\mathcal{H})_\mathcal{H}=f(z)\quad(f\in\mathcal{O}(\mathcal{D}),z\in\mathcal{D}),
\end{align*}
where we put $K_z^\mathcal{H}:=K^\mathcal{H}(\cdot,z)\in\mathcal{O}(\mathcal{D})$, owing to Riesz's representation theorem.
\end{remark}
Let $\Lambda$ be a Hausdorff space.
An injective continuous map
\begin{align*}
\Lambda\ni\lambda\mapsto K^\lambda\in\ext (\Gamma_G(\mathcal{D}))
\end{align*}
is called an {\it admissible parametrization} of $\ext (\Gamma_G(\mathcal{D}))$ if one has
\begin{align*}
\ext (\Gamma_G(\mathcal{D}))=\{0\}\coprod\coprod_{\lambda\in\Lambda}\mathbb{R}_{>0}K^\lambda,
\end{align*}
and the inverse map is universally measurable.
It is known that the latter condition automatically holds if $\Lambda$ is a locally compact second countable space.
We fix such an parametrization.
Then we have the following theorems.
\begin{theorem}{\cite[Theorem 1]{faraut_invariant_1999}}\label{th:integralexpression}
For any $K\in\Gamma_G(\mathcal{D})$, there exists a Radon measure $m$ on $\Lambda$ such that
\begin{equation}\label{eq:integralexpression}
K(z,z')=\int_\Lambda K^\lambda(z,z')\,dm(\lambda)\quad(z,z'\in\mathcal{D}).
\end{equation}
Here, the integral converges uniformly on compact sets in $z$ and $z'$.
\end{theorem}
\begin{theorem}{\cite[Theorem 2]{faraut_invariant_1999}}\label{th:uniquemeasure}
The following conditions are equivalent:
\begin{enumerate}[label=\textup{(\roman*)}]
\item $(\pi_0,\mathcal{O}(\mathcal{D}))$ is multiplicity-free;
\item For any $K\in\Gamma_G(\mathcal{D})$, the Radon measure $m$ giving the integral expression \eqref{eq:integralexpression} is unique.
\item Any unitary representation of $G$ realized in $\mathcal{O}(\mathcal{D})$ is multiplicity-free.
\end{enumerate}
\end{theorem}

\subsection{Orbit method}\label{subsect:orbit}
In this subsection, we explain a fundamental framework for the study of unitary representations of nilpotent Lie groups. 
We introduce a formula for the irreducible decompositions arising from the restrictions to subgroups, which will be applied to the representation on the space of all $L^2$ holomorphic functions in Sect. \ref{sect:L2} (see Proposition \ref{prop:234} for another application of this formula).

Now we assume that $G$ is a connected and simply connected nilpotent Lie group.
Then the unitary dual $\widehat{G}$ can be identified with the set of all coadjoint orbits in $\mathfrak{g}^*$ by the Kirillov-Bernat map \cite{bernat_representations_1972}, which will be denoted by
\begin{align*}
\widehat{\rho_G}:\mathfrak{g}^*\rightarrow\widehat{G}.
\end{align*}
Let $H=\exp\mathfrak{h}\subset G$ be an analytic subgroup, and $p:\mathfrak{g}^*\rightarrow\mathfrak{h}^*$ the canonical projection.
Let $n:\widehat{G}\times\widehat{H}\rightarrow\mathbb{N}\cup\{\infty\}$ be the Corwin-Greenleaf multipictiy-function given by
\begin{align*}
n_\pi(\nu):=\#\{H\text{-orbits in }\widehat{\rho_G}^{-1}(\pi)\cap p^{-1}(\widehat{\rho_H}^{-1}(\nu))\}.
\end{align*}
Let $m$ be the pushforward measure, by $\widehat{\rho_H}\circ p$, of a finite measure on $\mathfrak{g}^*$ equivalent to the $G$-invariant measure on $\widehat{\rho_G}^{-1}(\pi)$.
Then we have the following formula.
\begin{theorem}[{\cite{corwin_spectrum_1988}}]\label{th:250}
For $\pi\in\widehat{G}$, one has
\begin{align*}
\pi|_H\simeq \int_{\widehat{H}}^\oplus n_\pi(\nu)\nu\,dm(\nu).
\end{align*}
\end{theorem}
In the following, $n_\pi(\nu)$ may be simply denoted by $n(\nu)$ when $\pi$ is clear from the context.

\subsection{Self-dual homogeneous cone}\label{subsect:selfdual}
In this subsection, we review generalities about a self-dual homogeneous cone, one of the building blocks of a quasi-symmetric Siegel domain.
We focus particularly on its algebraic structure, namely, Jordan algebra.
We present in Theorem \ref{th:228} a spectral theorem, which will be used in Sect. \ref{sect:quasisymmetric}, specifically in Proposition \ref{prop:109}.
Also, it will be used to derive spectral decompositions of self-adjoint operators on $(V,h)$ in Sect. \ref{sect:parametrization}.
For a more detailed and comprehensive treatment of Jordan algebras, the reader is referred to \cite[Chapter I, §6-8]{satake_algebraic_2014} and \cite[Chapter II]{faraut_analysis_1994}.

Let $U$ be an $N$-dimensional vector space over $\mathbb{R}$, $\Omega$ a non-empty open convex cone in $U$, and assume that $\Omega$ is {\it regular}, that is, $\Omega$ does not contain any straight line.
The group
\begin{align*}
G(\Omega):=\{g\in GL(U)\mid g\Omega=\Omega\}
\end{align*}
has a natural structure of a Lie group since $G(\Omega)$ is closed in $GL(U)$.
We assume that $\Omega$ is {\it homogeneous}, i.e., the action of $G(\Omega)$ is transitive.
Fix an inner product $\langle\cdot,\cdot\rangle_U$ on $U$, and put
\begin{align*}
\Omega^*:=\{x\in U\mid\langle x, y\rangle_U>0\text{ for all }y\in\cl (\Omega)\backslash\{0\}\}.
\end{align*}
We also assume that $\Omega$ is {\it self-dual}, i.e., $\Omega=\Omega^*$.
Then it is known that there exists an $\mathbb{R}$-group $G$ in $GL(U)$, which contains $G(\Omega)$, and the Zariski component $G^z$ is a reductive $\mathbb{R}$-group with Cartan involution $g\mapsto{}^tg^{-1}$, where for $A\in\mathfrak{gl}(U)$, we denote by ${}^tA$ the adjoint of $A$ with respect to $\langle\cdot,\cdot\rangle_U$.
Let
\begin{align*}
\mathfrak{g}(\Omega)=\mathfrak{k}_0+\mathfrak{p}_0
\end{align*}
be the Cartan decomposition of $\mathfrak{g}(\Omega)$ corresponding to the Cartan involution
\begin{align*}
\theta:\mathfrak{g}(\Omega)\ni A\mapsto -{}^tA\in\mathfrak{g}(\Omega).
\end{align*}
Take a reference point $e\in U$ {\it compatible} with $\langle\cdot,\cdot\rangle_U$, namely, the following equivalence holds true:
\begin{equation*}
A\in\mathfrak{k}_0\Leftrightarrow Ae=0\quad(A\in\mathfrak{g}(\Omega)).
\end{equation*}
For $x\in U$, let $T_x\in\mathfrak{p}_0$ be the unique element satisfying $T_xe=x$.
Letting
\begin{align*}
xy:=T_xy\quad(x,y\in U),
\end{align*}
we obtain a unital Jordan algebra $(U,e)$, which implies that
\begin{equation}\label{eq:229}
[T_a,T_{bc}]+[T_b,T_{ca}]+[T_c,T_{ab}]=0\quad(a,b,c\in U).
\end{equation}
The trace form $\tau$ on $U$ is defined as follows:
\begin{align*}
\tau(x,y):=\tr T_{xy}\quad(x,y\in U).
\end{align*}
\begin{remark}\label{rem:254}
The adjoints of $A\in\mathfrak{g}(\Omega)$ with respect to $\tau$ and $\langle\cdot,\cdot\rangle_U$ coincide.
\end{remark}
For $x\in U$, put
\begin{align*}
P(x):=2(T_x)^2-T_{x^2}\in\mathfrak{gl}(U).
\end{align*}
\begin{definition}
We say $x\in U$ is {\it invertible} if $P(x)$ is non-singular, and in this case let
\begin{align*}
x^{-1}:=P^{-1}(x)x.
\end{align*}
The set of all invertible elements in $U$ is denoted by $U^\times$.
\end{definition}
It is known that for $x\in U^\times$, we have
\begin{align*}
xx^{-1}=e,\quad x^{-1}\in U^\times,\quad(x^{-1})^{-1}=x.
\end{align*}
\begin{definition}
Idempotents $c_1,c_2,\cdots,c_r$ of $U$ are called a {\it complete system of orthogonal idempotents} if one has
\begin{align*}
c_kc_l=0\quad(1\leq k\neq l\leq r),\quad e=\sum_{k=1}^r c_k.
\end{align*}
\end{definition}
The following result serves as a fundamental tool in this paper, particularly for handling non-invertible elements.
\begin{theorem}[{\cite[THEOREM III. 1.1]{faraut_analysis_1994}}]\label{th:228}
For any $x\in U$, there exists a complete system of orthogonal idempotents $c_1,c_2\cdots,c_r$ such that $x$ is expressed as
\begin{align*}
x=\sum_{k=1}^r\lambda_k c_k
\end{align*}
with $\lambda_k\in\mathbb{R}\,(1\leq k\leq r)$.
\end{theorem}
\begin{remark}\label{primitive}
In the above theorem, we may assume $c_1,c_2,\cdots, c_r$ are primitive idempotents, in other words, it is a {\it Jordan frame}, and $r$ equals the rank of $U$, and in this case, we have $\det x=\prod_{k=1}^r \lambda_k$ (see {\cite[THEOREM III. 1.2]{faraut_analysis_1994}}).
\end{remark}

\section{Quasi-symmetric Siegel domain}\label{sect:quasisymmetric}
A quasi-symmetric Siegel domains is defined via a Jordan algebra representation, and hence is
fundamentally grounded in the theory of self-dual homogeneous cones.
In this section, we begin by introducing the notion of the quasi-symmetric Siegel domain in a form suitable for our study.
For the original definition and classifications, see \cite[Chapter V]{satake_algebraic_2014}.
We then present several useful lemmas and an important proposition for the subsequent analysis, particularly in Sect. \ref{sect:parametrization}.
These lemmas are also applied in Sections \ref{sect:L2} and \ref{sect:coisotropic}.

Let $V$ be a finite dimensional vector space over $\mathbb{C}$ and $h$ a Hermitian inner product on $V$.
We denote by $j$ the complex structure on $V$.
Let $\beta:\mathfrak{g}(\Omega)\rightarrow\mathfrak{gl}(V)$ be a representation of the Lie algebra $\mathfrak{g}(\Omega)$ satisfying for $A\in\mathfrak{g}(\Omega)$ and $x\in U$,
\begin{align}
\beta(T_{Ax})&=\beta(A)\beta(T_x)+\beta(T_x)\beta(A)^*,\label{eq:230}\\
\beta({}^tA)&=\beta(A)^*,\label{eq:215}\\
\beta({\Id}_U)&=\frac{1}{2}{\Id}_V,\label{eq:214}
\end{align}
where for $B\in\mathfrak{gl}(V)$, we denote by $B^*$ the adjoint of $B$ with respect to $h$.
Put
\begin{align*}
R_x:=\beta(T_x)\quad(x\in U).
\end{align*}
Let $\mathcal{H}(V,h)$ be the set of all self-adjoint operators on $(V,h)$, and for $A\in\mathcal{H}(V,h)$, define $T_A\in\mathfrak{gl}(\mathcal{H}(V,h))$ by
\begin{align*}
T_AB:=\frac{1}{2}(AB+BA)\quad(A,B\in\mathcal{H}(V,h)),
\end{align*}
which induces on $\mathcal{H}(V,h)$ a structure of Jordan algebra.
It is know that the correspondence
\begin{align*}
(U,e)\ni x\mapsto 2R_x\in\mathcal{H}(V,h)
\end{align*}
is a unital Jordan algebra homomorphism, and $R_x$ is invertible when $x\in U^\times$.
We note that
\begin{equation}\label{eq:raxy}
\tr R_{(Ax)y}=\tr R_{x({}^tAy)}\quad(x,y\in U).
\end{equation}
Define a Hermitian map $Q:V\times V\rightarrow U_\mathbb{C}$ by
\begin{align*}
2h(R_xv,v')=\langle x,Q(v,v')\rangle_U\quad (x\in U, v,v'\in V),
\end{align*}
where $\langle\cdot,\cdot \rangle_U$ is extended to a $\mathbb{C}$-bilinear form.
Then it is known that $Q$ is $\Omega$-{\it positive}, i.e.,
\begin{align*}
Q(v,v)\in\cl (\Omega)\backslash\{0\}\quad(v\neq 0).
\end{align*}
\begin{definition}
We call the following domain a {\it quasi-symmetric Siegel domain}:
\begin{align*}
{\mathcal{S}(\Omega,Q)}:=\{(z,v)\in U_\mathbb{C}\times V\mid \im z-Q(v,v)\in\Omega\}.
\end{align*}
\end{definition}
It is known that for $A\in\mathfrak{g}(\Omega)$, if we regard $e^A\in GL(U)$ as a $\mathbb{C}$-linear map from $U_\mathbb{C}$ to itself by the complex linearity, then we have
\begin{equation}\label{eq:232}
e^AQ(v,v')=Q(e^{\beta(A)}v,e^{\beta(A)}v')
\end{equation}
and hence the map
\begin{align*}
t(A):=(e^A,e^{\beta(A)})\in GL(U_\mathbb{C})\times GL(V)
\end{align*}
preserves ${\mathcal{S}(\Omega,Q)}$. 
We shall see some useful equalities for studying the multiplicity-freeness property of group representations and the coisotropicity of group actions.
\begin{lemma}\label{lem:105}
For $x\in U$, $v\in V$, and $A\in\mathfrak{g}(\Omega)$, let
\begin{align*}
\tilde{x}:=e^Ax,\quad \tilde{v}:=e^{-\beta(A)^*}v.
\end{align*}
Then one has
\begin{align*}
e^{\beta(A)}R_xv=R_{\tilde{x}}\tilde{v}.
\end{align*}
\end{lemma}
\begin{proof}
The result follows from
\begin{align*}\begin{split}
2h(R_{\tilde{x}}\tilde{v},w)&=\langle\tilde{x},Q(\tilde{v},w)\rangle_U
\\&=\langle x,e^{{}^tA}Q(\tilde{v},w)\rangle_U
\\&=\langle x,Q(v,e^{\beta(A)^*}w)\rangle_U
\\&=2h(R_xv,e^{\beta(A)^*}w)
\\&=2h(e^{\beta(A)}R_xv,w)\quad(w\in V).
\end{split}
\end{align*}
Here, for the third equality, we have used \eqref{eq:215} and \eqref{eq:232}.
\end{proof}
For $x\in U$, let $g_x\in \Sym^2((V_\mathbb{R})^*)$ be given by
\begin{align*}
g_x(v_1,v_2):=\langle x,\re Q(v_1,v_2)\rangle_U\quad(v_1,v_2\in V).
\end{align*}
For a vector space $V_0$ over $\mathbb{R}$, a symmetric bilinear form $b$ on $V_0$, and a subspace ${W}\subset V_0$, put
\begin{align*}
{W}^{\perp,\,b}:=\{v\in V_0\mid b(v,w)=0\text{ for all }w\in {W}\}.
\end{align*}
\begin{lemma}\label{lem:111}
For $x\in U$ and $y\in U^\times$, one has
\begin{align*}
(({W}^{\perp,\,g_x})^{\perp,\,g_y})^{\perp,\,g_x}={W}^{\perp,\, g_{P(x)y^{-1}}}.
\end{align*}
\end{lemma}
\begin{proof}
For $v\in V$, the following equivalences hold true:
\begin{align*}
v\in({W}^{\perp,\,g_x})^{\perp,\,g_y}
\Leftrightarrow \re h(R_yv,{W}^{\perp,\,g_x})=\{0\}\Leftrightarrow R_yv\in({W}^{\perp,\,g_x})^{\perp,\,\re h}.
\end{align*}
Hence one has
\begin{align*}
({W}^{\perp,\,g_x})^{\perp,\,g_y}=R_{y^{-1}}({W}^{\perp,\,g_x})^{\perp,\,\re h},
\end{align*}
and the following equivalences hold true:
\begin{align*}
v\in(({W}^{\perp,\,g_x})^{\perp,\,g_y})^{\perp,\,g_x}
&\Leftrightarrow \re h(R_xv,({W}^{\perp,\,g_x})^{\perp,\,g_y})=\{0\}
\\&\Leftrightarrow\re h(R_xv,R_{y^{-1}}({W}^{\perp,\,g_x})^{\perp,\, \re h})=\{0\}
\\&\Leftrightarrow \re h(R_{y^{-1}}R_xv,({W}^{\perp,\,g_x})^{\perp,\,\re h})=\{0\}
\\&\Leftrightarrow R_{y^{-1}}R_xv\in {W}^{\perp,\,g_x}
\Leftrightarrow \re h(R_xR_{y^{-1}}R_xv,{W})=\{0\}.
\end{align*}
Whereas we have
\begin{align*}
R_xR_{y^{-1}}R_x=R_{P(x)y^{-1}},
\end{align*}
which gives the desired equality.
\end{proof}
Next proposition plays a fundamental role for the subsequent analysis.
\begin{proposition}\label{prop:109}
For a real subspace ${W}\subset V$, if 
\begin{align*}
\im Q({W}^{\perp,\,g_e},{W}^{\perp,\,g_e})=\{0\},
\end{align*}
then one has
\begin{align*}
\langle x,\im Q({W}^{\perp,\,g_x},{W}^{\perp,\,g_x})\rangle_U=\{0\}\quad(x\in U).
\end{align*}
\end{proposition}
\begin{proof}
Since
\begin{align*}
\langle y,\re Q(j{W}^{\perp,\,g_e},{W}^{\perp,\,g_e})\rangle_U=\{0\}\quad(y\in U^\times),
\end{align*}
we have
\begin{align*}
j{W}^{\perp,\,g_{y^{-1}}}\subset {W}\quad(y\in U^\times)
\end{align*}
by Lemma \ref{lem:111}.
Furthermore, it follows that for any $x\in U$ and $y\in U^\times$,
\begin{align*}
j{W}^{\perp,\,g_x}&\subset({W}^{\perp,\,g_{y^{-1}}})^{\perp,\,g_x}
\\&\subset (({W}+\ker (g_x))^{\perp,\,g_{y^{-1}}})^{\perp,\,g_x}
\\&=((({W}^{\perp,\,g_x})^{\perp,\,g_x})^{\perp,\,g_{y^{-1}}})^{\perp,\,g_x}=({W}^{\perp,\,g_x})^{\perp,\,g_{P(x)y}}.
\end{align*}
Let $x=\sum_{k=1}^r\lambda_kc_k$ be a decomposition in Theorem \ref{th:228}.
For $1\leq k\leq r$, let
\begin{align*}
\widetilde{\lambda_k}:=\begin{cases}
\lambda_k^{-1}&(\lambda_k\neq 0)\\
1&(\lambda_k=0)\end{cases},\quad y:=\sum_{k=1}^r\widetilde{\lambda_k}c_k.
\end{align*}
Here we may assume that $y\in U^\times$ (see Remark \ref{primitive}).
Then we have $P(x)y=x$.
Hence we obtain
\begin{align*}
j{W}^{\perp,\,g_x}\subset({W}^{\perp,\,g_x})^{\perp,\,g_x},
\end{align*}
which proves the assertion.
\end{proof}

\section{\texorpdfstring{Multiplicity-freeness and the vanishing condition for $\im Q$}{Multiplicity-freeness and the vanishing condition for Im Q}}\label{sect:parametrization}
In this section, we show \ref{condi:mainreal} $\Rightarrow$ \ref{condi:mainmf} of Theorem \ref{th:main} and describe an admissible parametrization of $\ext(\Gamma_{G^W}(\mathcal{S}(\Omega,Q)))$.
\subsection{Description of extremal invariant reproducing kernels}\label{subsect:extremal}
In this subsection, we determine all extremal $G^W$-invariant reproducing kernels.
Such functions are partially labeled by vectors in $U$ by the Kirillov-Bernat mapping.
In Lemmas \ref{lem:203} and \ref{lem:101}, we deduce constraints on the vectors in $U$ using the ideas of the coherent state representations as a guiding concept.
From these ideas, we also obtain a differential equation \eqref{eq:235}, which encodes significant information of the reproducing kernels (see Proposition \ref{prop:202} below), though holomorphic functions over a complex vector subspace remain undetermined.
To determine the functions, we provide a suitable decomposition \eqref{eq:224} of $V$ under the assumption \eqref{condi:key} below, and we get the full expression of the reproducing kernels in Theorem \ref{th:4.4}.
Next, we establish an admissible parametrization of $\ext(\Gamma_{G^W}(\mathcal{S}(\Omega, Q)))$, with particular attention to its continuity, as presented in Corollary \ref{cor:115}.

For $x_0\in U$, $v_0\in V$, let ${\bf n}(x_0,v_0):{\mathcal{S}(\Omega,Q)}\rightarrow{\mathcal{S}(\Omega,Q)}$ be the affine transformation of ${\mathcal{S}(\Omega,Q)}$ defined by
\begin{align*}
{\bf n}(x_0,v_0)(z,v):=(z+x_0+2iQ(v,v_0)+iQ(v_0,v_0), v+v_0),
\end{align*}
and for a real subspace ${W}\subset V$, put
\begin{align*}
G^{W}:=\{{\bf n}(x_0,v_0)\mid x_0\in U, v_0\in {W}\},
\end{align*}
which has a natural structure of a Lie group.
In what follows, $\mathfrak{g}^{W}$ may be naturally identified with $U\oplus {W}$.
Also, we may identify other related vector spaces, such as the complexifications, the dual spaces and so on.
We shall see the group law and some related formulae.
For $(x_1,v_1),(x_2,v_2)\in U\oplus V$, we have
\begin{align*}
{\bf n}(x_1,v_1){\bf n}(x_2,v_2)&={\bf n}(x_1+x_2+\frac{1}{2}[v_1,v_2],v_1+v_2),\\
\exp(x_1,v_1)&={\bf n}(x_1,v_1),\\
\Ad ({\bf n}(x_1,v_1))(x_2,v_2)&=(x_2+[v_1,v_2],v_2),\\
{[v_1,v_2]}&=4\im Q(v_1,v_2).
\end{align*}

Suppose that $(\pi,\mathcal{H})\in\widehat{G^W}$ is realized in $\mathcal{O}(\mathcal{S}(\Omega,Q))$ and corresponds to the coadjoint orbit through $(-\nu)\in(\mathfrak{g}^{W})^*$ by the Kirillov-Bernat map.
Let $K^\mathcal{H}\in\Gamma_{G^{W}}({\mathcal{S}(\Omega,Q)})$ be the reproducing kernel of $\mathcal{H}$.
Put
\begin{align*}
G_1:=\{{\bf n}(x,0)\mid x\in U\}.
\end{align*}
The proof of the next proposition is omitted, as it is essentially the same as that provided in our previous paper.
\begin{proposition}[{\cite[Proposition 4.4]{arashi_multiplicityfree_a}}]\label{prop:234}
One has
\begin{align*}
K_{(ie,0)}^\mathcal{H}(z,v)=e^{i\langle\nu,z\rangle}F(v)\quad((z,v)\in{\mathcal{S}(\Omega,Q)})
\end{align*}
for some $F\in\mathcal{O}(V)$.
\end{proposition}
Let ${P}:={W}\cap j{W}$.
The reproducing kernel $K^\mathcal{H}$ is $G^{W}$-invariant, and hence one has
\begin{equation}\label{eq:233}
d\pi_0(a)K_{(ie,0)}^\mathcal{H}=0
\end{equation}
for $a=q-ijq$ with $q\in {P}$, where we abbreviate $\pi_0|_\mathcal{H}$ to $\pi_0$ and extend the differential representation $d\pi_0$ to a representation of $(\mathfrak{g}^{W})_\mathbb{C}$ by the complex linearity.
Let $\mathfrak{h}_-$ be the complex subalgebra of $\mathfrak{g}_\mathbb{C}$ given by
\begin{equation*}
U_\mathbb{C}\oplus\{q+ijq\mid q\in {P}\}.
\end{equation*}
Define $\widetilde{\nu}\in\left(\mathfrak{g}^W\right)^*$ by
\begin{align*}
\langle\widetilde{\nu},x+v\rangle:=\langle {\nu},x\rangle\quad (x\in U,v\in W).
\end{align*}
By \eqref{eq:233} and Proposition \ref{prop:234}, we see that $f:=K_{(ie,0)}^{\mathcal{H}}$ solves the following equation:
\begin{equation}\label{eq:235}
d\pi_0(a)f=-i\langle \widetilde{\nu},a\rangle\quad (a\in\overline{\mathfrak{h}_-}).
\end{equation}
Let $S$ be a real subspace of $V$ such that
\begin{equation}\label{eq:221}
V={P}\oplus S\oplus jS.
\end{equation}
\begin{remark}\label{decomposition242}
\eqref{eq:221} is satisfied when we have
\begin{align*}
{W}=S\oplus {P},\quad {W}+j{W}=V.
\end{align*}
\end{remark}

Now we shall find a solution of \eqref{eq:235}.
For a $U_\mathbb{C}\times V$-valued function $X$ on ${\mathcal{S}(\Omega,Q)}$, we define an operator $D_X$ on $C^\infty(\mathcal{S}(\Omega,Q))$ by
\begin{align*}
D_Xf_0({\bf z}):=\frac{d}{dt}\Bigr|_{t=0}f_0({\bf z}+tX({\bf z}))\quad({\bf z}\in{\mathcal{S}(\Omega,Q)}).
\end{align*}
For $a=q_0-ijq_0$ with $q_0\in {P}$, the equation \eqref{eq:235} tells us that
\begin{align*}
\left(D_{(-2iQ(v,q_0),-q_0)}-iD_{(-2iQ(v,jq_0),-jq_0)}\right)f(z,v)=0.
\end{align*}
Since $f$ is holomorphic in $z$, and $Q$ is Hermitian, it follows that
\begin{align*}
\left(D_{(0,q_0)}-iD_{(0,jq_0)}\right)f(z,v)=0,
\end{align*}
which implies that
\begin{equation}\label{eq:236}
f(z,q+s)=e^{i\langle\nu,z\rangle}F(q+s)=e^{i\langle\nu,z\rangle}F(s)\quad(q\in {P},s\in S\oplus jS,(z,q+s)\in{\mathcal{S}(\Omega,Q)}).
\end{equation}
Now the $G^W$-invariance of $K^\mathcal{H}$ yields the following proposition.
\begin{proposition}\label{prop:202}
One has
\begin{align*}
K^\mathcal{H}(z,q+s,z',q'+s')=\exp({i\langle\nu,z-\overline{z'}-2iQ(q+s,q'+s')\rangle})F(s-\overline{s'})
\end{align*}
with $q,q'\in{P},s,s'\in S\oplus jS$.
\end{proposition}
\begin{proof}
For $x_0\in U$ and $v_0\in W$, we have
\begin{align*}
{\bf n}(x_0,v_0)(ie,0)&=(ie+x_0+iQ(v_0,v_0),v_0),\\
{\bf n}(-x_0,-v_0)(z,v)&=(z-x_0+2iQ(v,-v_0)+iQ(v_0,v_0),v-v_0).
\end{align*}
Hence we have
\begin{align*}
K^\mathcal{H}(z,q+s,&ie+x_0+iQ(v_0,v_0),v_0)
\\&=K^\mathcal{H}(z-x_0+2iQ(q+s,-v_0)+iQ(v_0,v_0),v-v_0,ie,0)
\\&=\exp({i\langle\nu,z-x_0+2iQ(q+s,-v_0)+iQ(v_0,v_0)\rangle})F(s-s_0),
\end{align*}
where $v_0=q_0+s_0$ with $q_0\in{P}$ and $s_0\in S$.
The last expression equals
\begin{align*}
\exp({i\langle\nu,z-\overline{(x_0+ie+iQ(v_0,v_0)})-2iQ(q+s,v_0)\rangle})F(s-s_0)
\end{align*}
up to a constant.
By the analytic continuation, we get the desired expression.
\end{proof}

We now proceed to establish the decomposition \eqref{eq:224} in order to determine the function $F$ in Proposition \ref{prop:202}.
In the subsequent argument, it will be necessary to choose a suitable complementary subspace of $P$ within $W$, which is orthogonal to $jW$ with respect to $g_\nu$.
We address this problem in two ways.
First, we consider the constraints \eqref{positivity238} and \eqref{eq:113} on $\nu\in U$ below.
Second, we show that if \ref{condi:mainreal} of Theorem \ref{th:main} holds, then the relation $W\cap j W=P$ between $W$, $jW$, and $P$ is preserved when considering their images under the projection from $V$ to $V/\ker(g_x)$ (see Proposition \ref{prop:107}).
Combining these facts leads to the desired decomposition.

We assume that $\mathcal{H}\neq 0$.
Then Proposition \ref{prop:202} tells us that $F\neq 0$, and hence $f\neq 0$.
The following lemma is greatly influenced by the study of coherent state representation.
In what follows, we may abbreviate $\nu|_{\mathfrak{g}_1}$ to $\nu$ if there is no confusion, and regard $\nu\in U^*$ as the vector in $U$ by means of $\langle\cdot,\cdot\rangle_U$.
\begin{lemma}[c.f. {\cite[2. Proposition]{lisiecki_kaehler_1990}}]\label{lem:203}
For ${x}=\nu\in U$, we have
\begin{equation}\label{positivity238}
-i\langle{x},[a,\overline{a}]\rangle_U=8\langle{x},Q(q,q)\rangle_U\geq 0\quad (a=q+ijq,q\in{P}).
\end{equation}
\end{lemma}
\begin{proof}
Let us consider the projective space $\mathbb{P}(\mathcal{H})$ and the natural projection $p_f$ between the tangent spaces $T_f\mathcal{H}$ and $T_{[f]}\mathbb{P}(\mathcal{H})$.
The Fubini-Study metric on $T_{[f]}\mathbb{P}(\mathcal{H})$ satisfies
\begin{align*}
\|dp_f(f_0)\|_{FS}^2=\frac{\|f\|_\mathcal{H}^2\|f_0\|_\mathcal{H}^2-|(f,f_0)_\mathcal{H}|^2}{\|f\|_\mathcal{H}^4},
\end{align*}
and \eqref{positivity238} can be derived from \eqref{eq:235} and its consequence:
\begin{equation}\label{eq:201}
\begin{split}
\|dp_f(d\pi_0(a)f)\|_{FS}^2=\frac{(d\pi_0([a,\overline{a}])f,f)_\mathcal{H}}{\|f\|_\mathcal{H}^2}=-i\langle\nu,[a,\overline{a}]\rangle_U\quad(a=q+ijq,q\in{P}).
\end{split}\end{equation}
\end{proof}
For $x\in U$, define a complex linear subspace $N_{x}\subset P$ by
\begin{align*}
N_{x}:=\{q\in{P}\mid g_{x}(q,q)=0\}.
\end{align*}
\begin{lemma}\label{lem:101}
For ${x}=\nu\in U$, one has
\begin{equation}\label{eq:113}
N_{x}\subset\ker (g_{x}).
\end{equation}
\end{lemma}
\begin{proof}
When $\langle\nu,Q(q_0,q_0)\rangle_U=0$ with $q_0\in{P}$, we see from \eqref{eq:201} that
\begin{align*}
d\pi_0(a_0)f(z,v)=\zeta f(z,v)\quad (a_0=q_0+ijq_0)
\end{align*}
for some $\zeta\in\mathbb{C}$.
By \eqref{eq:236}, we have
\begin{align*}
d\pi_0(a_0)f(z,v)=4\langle\nu,Q(v,q_0)\rangle_U f(z,v),
\end{align*}
and hence
\begin{equation}\label{eq:239}
(-4\langle\nu,Q(v,q_0)\rangle_U+\zeta)F(s)=0.
\end{equation}
Suppose that
\begin{align*}
-4\langle\nu,Q(v_0,q_0)\rangle_U+\zeta\neq 0
\end{align*}
for some $v_0\in V$.
Then there exists an open subset $\widetilde{V}\subset V$ such that
\begin{align*}
-4\langle\nu,Q(v,q_0)\rangle_U+\zeta\neq 0\quad(v\in \widetilde{V}),
\end{align*}
and hence \eqref{eq:239} implies that $F(s)\equiv0$, and contradicts that $\mathcal{H}\neq \{0\}$.
Hence the assertion holds.
\end{proof}
For $x\in U$, $v\in V$, and a real subspace $V_0\subset V$, define
\begin{align*}
[v]_x:&=v+\ker (g_x)\quad\in V/\ker (g_x),\\
{[V_0]_x}:&=V_0+\ker (g_x)/\ker (g_x)\quad\subset V/\ker (g_x).
\end{align*}
Then we have the following lemma.
\begin{lemma}\label{lem:241}
One has $[{P}]_{\nu}\cap[jW^{\perp,\,g_{\nu}}]_{\nu}=\left\{\ker (g_{\nu})\right\}$.
\end{lemma}
\begin{proof}
If $v\in jW^{\perp,\,g_{\nu}}$ satisfies $[v]_{\nu}\in[{P}]_{\nu}$, then there exsists $w\in\ker (g_{\nu})$ such that 
\begin{align*}
v-w\in{P}\cap jW^{\perp,\,g_{\nu}},
\end{align*}
and we see from Lemma \ref{lem:101} that $v\in\ker (g_{\nu})$, which proves the assertion.
\end{proof}

The following theorem is a key component in the proof of the next lemma and proposition.
\begin{theorem}\label{th:103}
For any $x\in U$, $v\in V$, and a real subspace $S\subset V$ satisfying $\im Q(S,S)=\{0\}$, suppose that $R_xv\in S\oplus jS$, then there exists $w\in V$ such that $\overline{R_xv}^S=R_xw$.
\end{theorem}
\begin{proof}
Let $x=\sum_{k=1}^r\lambda_kc_k$ be the decomposition in Theorem \ref{th:228}.
Letting $p_k:=2R_{c_k}\,(1\leq k\leq r)$, we have
\begin{align*}
p_1+p_2+\cdots +p_r={\Id}_V,\quad p_kp_l=\delta_{kl}p_k\quad(1\leq k,l\leq r),
\end{align*}
where the latter equality can be seen from \eqref{eq:229} and \eqref{eq:230} with $a$, $b$, $c$ replaced by $c_k$, $c_k$, $c_l$, respectively.
We may assume that $\lambda_k=0$ if and only if $1\leq k\leq r_0$, and put
\begin{align*}
{x'}:=\sum_{k=1}^{r_0}c_k,\quad R_xv=s_1+js_2
\end{align*}
with $s_1,s_2\in S$.
Then we have
\begin{align*}
\langle{x'},Q(s_1+js_2,s_m)\rangle_U=2h(R_{{x'}}R_xv,s_m)=0\quad(m=1,2).
\end{align*}
Since $\im Q(S,S)=\{0\}$, this implies that
\begin{align*}
\begin{split}
0=\langle{x'},Q(s_m,s_m)\rangle_U=2h(R_{{x'}}s_m,s_m)=h\left(\sum_{k=1}^{r_0}p_ks_m,s_m\right)=\sum_{k=1}^{r_0}h(p_ks_m,p_ks_m),
\end{split}
\end{align*}
and hence $p_ks_m=0\,(1\leq k\leq r_0,m=1,2)$.
Now if we put
\begin{align*}
w_m:=2\sum_{k=r_0+1}^r\lambda_k^{-1}p_ks_m\quad(m=1,2),
\end{align*}
then we have
\begin{align*}
s_m=\sum_{k=r_0+1}^rp_ks_m=R_xw_m,
\end{align*}
and thus
\begin{align*}
\overline{R_xv}^S=s_1-js_2=R_x(w_1-jw_2),
\end{align*}
which completes the proof.
\end{proof}
Put $S:=jW^{\perp,\,\re h}$.
For ${x}\in U$, let
\begin{align*}
S_{x}:=jW^{\perp,\,g_{x}}\cap W.
\end{align*}
We may use the notation $\mathfrak{g}^W({x}):=\mathfrak{g}_1\oplus S_{x}$, which is consistent with the one in the orbit method if we regard $x$ as an element of ${(\mathfrak{g}^W)}^*$ by means of $\langle\cdot,\cdot\rangle_U$.
From here until the end of the next subsection, we assume that
\begin{align}\label{condi:key}
\im Q(S,S)=\{0\}.
\end{align}
\begin{lemma}\label{lem:106}
Let $s\in S\oplus jS$ and $x\in U$.
If there exists $q\in{P}$ such that $s+q\in\ker (g_x)$, then there exists $q'\in{P}$ such that $\overline{s}^S+q'\in\ker (g_x)$.
\end{lemma}
\begin{proof}
Note that
\begin{align*}
v\in\ker (g_x)\Leftrightarrow R_xv=0
\end{align*}
for $v\in V$.
So the following conditions for $s'\in S\oplus jS$ are equivalent:
\begin{enumerate}[label=({\roman*})]
\item
There exists $q'\in{P}$ such that $s'+q'\in\ker (g_x)$;
\item\label{cond:(ii)}
$h(R_xs',v)=0\text{ for all }v\in(R_x{P})^{\perp,\,\re h}$.
\end{enumerate}
Since
\begin{align*}
{P}^{\perp,\,\re h}=(W\cap jW)^{\perp,\,\re h}=W^{\perp,\,\re h}+jW^{\perp,\,\re h}=S\oplus jS,
\end{align*}
\ref{cond:(ii)} is equivalent to
\begin{enumerate}[label=({\roman*}'),start=2]
\item\label{cond:ii'}
For any $v\in V$, if $R_xv\in S\oplus jS$, then $h(R_xs',v)=0$.
\end{enumerate}
Now suppose that we have \ref{cond:ii'} with $s'=s$, and $R_xv\in S\oplus jS$.
Then by Theorem \ref{th:103}, we can find $w\in V$ such that
\begin{align*}
R_xw=\overline{R_xv}^S\in S\oplus jS,
\end{align*}
and hence $h(R_xs,w)=0$.
Since $\langle e,\im Q(S,S)\rangle_U=\{0\}$, we have
\begin{equation}\label{eq:240}
\re h(jS,S)=\{0\},
\end{equation}
and hence
\begin{align*}
h(R_x\overline{s}^S,v)&=h(\overline{s}^S,R_xv)
\\&=h(\overline{s}^S,\overline{R_xw}^S)
\\&=h(R_xw,s)
\\&=h(w,R_xs)=\overline{h(R_xs,w)}=0,
\end{align*}
which completes the proof.
\end{proof}
Note that \eqref{eq:240} shows
\begin{align*}
S\subset jS^{\perp,\,\re h}=W.
\end{align*}
Take a real form $P_0$ of $P$.
In view of the Remark \ref{decomposition242}, the subspace
\begin{align*}
W_0:=S+P_0\subset V
\end{align*}
is a real form of $V$.
For $v\in V$, let us denote
\begin{align*}
\re v:=\frac{1}{2}(v+\overline{v}^{W_0}),\quad \im v:=-\frac{1}{2}j(v-\overline{v}^{W_0}).
\end{align*}
Then $v\in W$ implies that $\im v\in {P}$.
\begin{proposition}\label{prop:107}
One has $[jW]_x\cap[W]_x=[{P}]_x\,(x\in U)$.
\end{proposition}
\begin{proof}
First we note that the following relation holds:
\begin{align*}
(\overline{\ker (g_x)}^{W_0}+{P})/{P}\subset(\ker (g_x)+{P})/{P}.
\end{align*}
Indeed, let $v=s+q\in\ker (g_x)$ with $s\in S\oplus jS$ and $q\in{P}$.
Then by Lemma \ref{lem:106}, we can find $q'\in{P}$ such that
\begin{align*}
\overline{s}^{W_0}+q'\in\ker (g_x).
\end{align*}
Now we can see the relation from
\begin{align*}
\overline{v}^{W_0}=\overline{s}^{W_0}+\overline{q}^{W_0}
=(\overline{s}^{W_0}+q')+(\overline{q}^{W_0}-q').
\end{align*}
Next, for $v_1,v_2\in W$, suppose that $jv_1-v_2\in\ker (g_x)$.
Then there exists $q\in{P}$ such that
\begin{align*}
jv_1-v_2-j\overline{v_1}^{W_0}-\overline{v_2}^{W_0}+q\in\ker (g_x),
\end{align*}
and hence
\begin{align*}
-2\im v_1-2\re v_2+q\in\ker (g_x).
\end{align*}
It is clear that
\begin{align*}
v_2-\frac{q}{2}=(j\im v_2-\im v_1)+(\re v_2+\im v_1-\frac{q}{2}),
\end{align*}
where the first term is contained in ${P}$ and the second term $\ker (g_x)$.
Hence we obtain
\begin{align*}
v_2+\ker (g_x)\in[{P}]_x,
\end{align*}
which proves the assertion.
\end{proof}
Combining Proposition \ref{prop:107} and Lemma \ref{lem:241}, we get the following theorem.
\begin{theorem}\label{th:108}
Suppose that $\im Q(S,S)=\{0\}$.
Then one has $W={P}+S_{\nu}$.
\end{theorem}
\begin{proof}
By Proposition \ref{prop:109}, we have $[jW^{\perp,\,g_{\nu}}]_{\nu}\subset[W]_{\nu},$ and hence
\begin{align*}
\left([W]_{\nu}\cap[jW]_{\nu}\right)+[jW^{\perp,\,g_{\nu}}]_{\nu}\subset[W]_{\nu}\subset V/\ker (g_{\nu}).
\end{align*}
Let $\overline{g_{\nu}}$ be the non-degenerate symmetric bilinear form on $[V]_{\nu}$ induced by $g_{\nu}$.
The dimension of
\begin{align*}
[jW]_{\nu}\cap[W]_{\nu}=[jW^{\perp,\,g_{\nu}}]_{{\nu}}^{\perp,\,\overline{g_{\nu}}}\cap[W]_{\nu}
\end{align*}
is greater than or equal to $\dim[W]_{\nu}-\dim[jW^{\perp,\,g_{\nu}}]_{\nu}$.
Then Proposition \ref{prop:107} and Lemma \ref{lem:241} tell us that
\begin{align*}
[{P}]_{\nu}\oplus[jW^{\perp,\,g_{\nu}}]_{\nu}=[W]_{\nu}\subset V/\ker (g_{\nu}),
\end{align*}
which proves ${P}+jW^{\perp,\,g_{\nu}}\supset W$ and hence the assertion.
\end{proof}
Let ${x}\in U$ be satisfying \eqref{positivity238} and \eqref{eq:113}.
Then we have
\begin{align*}
S_{x}\cap jS_{x}={P}\cap\ker (g_{x})=N_{x}.
\end{align*}
Choose any real subspace $S^{x}\subset S_{x}$ complementary to $N_{x}$.
Then we have
\begin{equation*}
S_{x}+jS_{x}=S^{x}\oplus jS^{x}\oplus N_{x}.
\end{equation*}
In the following, for $s\in S^{x}\oplus jS^{x}$, we abbreviate $\overline{s}^{S^{x}}$ to $\overline{s}^{x}$.
Letting
\begin{align*}
{P}^{x}:=N_{x}^{\perp,\,\re h}\cap{P},
\end{align*}
we have
\begin{equation}\label{eq:224}
W={P}+S_{x}={P}^{x}\oplus S_{x}={P}\oplus S^{x}
\end{equation}
in view of Theorem \ref{th:108}.

Now recall from Proposition \ref{prop:202} that
\begin{align*}
K^\mathcal{H}(z,q+s,z',q'+s')=\exp({i\langle\nu,z-\overline{z'}-2iQ(q,q')-2iQ(s,s')\rangle_U})F(s-\overline{s'}^{\nu})
\end{align*}
with $q,q'\in{P},s,s'\in S^{\nu}\oplus jS^{\nu}$ for some $F\in\mathcal{O}(S^{\nu}\oplus jS^{\nu})$ (see also Remark \ref{decomposition242}).
The following proposition plays a crucial role in the derivation of $K^\mathcal{H}$ in Theorem \ref{th:4.4} below.
\begin{proposition}\label{prop:223}
For any $H\in\Gamma(S^{\nu}\oplus jS^{\nu})$, the function
\begin{align*}
&\widetilde{H}(z,q+s,z',q'+s'):
\\&=\exp({\langle\nu,i(z-\overline{z'})+2Q(q,q')+Q(s,\overline{s}^\nu)+Q(\overline{s'}^\nu,s')\rangle_U})H(s,s')
\end{align*}
is contained in $\Gamma({\mathcal{S}(\Omega,Q)})$.
Moreover, if $H\in\Gamma_{S^{\nu}}(S^{\nu}\oplus jS^{\nu})$, then $\widetilde{H}\in\Gamma_{G^W}({\mathcal{S}(\Omega,Q)})$.
\end{proposition}
\begin{proof}
For $a_k\in\mathbb{C}$, $z_k\in U_\mathbb{C}$, $q_k\in{P}$, $s_k\in S^{\nu}\oplus jS^{\nu}\,(k=1,2,\cdots, n_0)$, we have
\begin{align*}
\sum_{k,l=1}^{n_0}&a_k\overline{a_l}\widetilde{H}(z_l,q_l+s_l,z_k,q_k+s_k)
\\=\sum_{k,l=1}^{n_0}&\exp({\langle\nu,i(z_l-\overline{z_k})+2Q(q_l,q_k)+Q(s_l,\overline{s_l}^\nu)+Q(\overline{s_k}^\nu,s_k)\rangle_U})
\\&\quad\cdot a_k\overline{a_l}H(s_l,s_k),
\end{align*}
which equals
\begin{align*}
\sum_{k,l=1}^{n_0}a_ke^{\langle\nu,Q(s_l,\overline{s_l}^{\nu})\rangle_U}e^{i\langle\nu,z_l\rangle_U}\overline{a_le^{\langle\nu,Q(s_k,\overline{s_k}^{\nu})\rangle_U} e^{i\langle\nu,z_k\rangle_U}}e^{2\langle\nu,Q(q_l,q_k)\rangle_U}H(s_l,s_k),
\end{align*}
by Proposition \ref{prop:109}.
Now the positivity of $\widetilde{H}$ can be seen from the expression and Lemma \ref{lem:203}.
The latter assertion immediately follows from the fact
\begin{align*}
\exp({\langle\nu,i(z-\overline{z'})+2Q(q,q')+2Q(s,s')\rangle_U})
\end{align*}
is $G^W$-invariant.
\end{proof}
\begin{theorem}\label{th:4.4}
There exists $\chi\in (S^{\nu})^*$ such that
\begin{equation}\label{eq:249}
\begin{split}
&K^\mathcal{H}(z,q+s,z',q'+s')
\\&=\exp({i\langle\nu,z-\overline{z'}-2iQ(q,q')-iQ(s,\overline{s}^\nu)-iQ(\overline{s'}^\nu,s')\rangle_U})
\\&\quad\cdot e^{-i\langle\chi,s-\overline{s'}^{\nu}\rangle}
\end{split}
\end{equation}
with $q,q'\in{P},s,s'\in S^{\nu}\oplus jS^{\nu}$.
\end{theorem}
\begin{proof}
By the positivity of $K^\mathcal{H}$, for any $a_k\in\mathbb{C}$, $s_k\in S^{\nu}\oplus jS^{\nu}\,(k=1,2,\cdots, n_0)$, we have
\begin{equation*}
\sum_{k,l=1}^{n_0}a_k\overline{a_l}e^{2\langle\nu,Q(s_l,s_k)\rangle_U}F(s_l-\overline{s_k}^{\nu})\geq 0,
\end{equation*}
and hence by Proposition \ref{prop:109}, the function
\begin{align*}
H_0(s,s'):=e^{-\langle\nu,Q(s-\overline{s'}^{\nu},\overline{s}^{\nu}-s')\rangle_U}F(s-\overline{s'}^{\nu})\quad(s,s'\in S^{\nu}\oplus jS^{\nu})
\end{align*}
is contained in $\Gamma_{S^{\nu}}(S^{\nu}\oplus jS^{\nu})$.
By Proposition \ref{prop:223}, it follows that the function $H_0\in\Gamma_{S^{\nu}}(S^{\nu}\oplus jS^{\nu})$ is extremal since $K^\mathcal{H}=\widetilde{H_0}$ is extremal.
The corresponding Hilbert subspace of $\mathcal{O}(S^{\nu}\oplus jS^{\nu})$ is irreducible by Theorem \ref{th:243}, and hence we get
\begin{align*}
H_0(s,s')=e^{-i\langle\chi,s-\overline{s'}^{\nu}\rangle}
\end{align*}
for some $\chi\in (S^{\nu})^*$.
\end{proof}
Next, we present a concrete description of an admissible parametrization of $\ext(\Gamma_{G^W}(\mathcal{S}(\Omega,Q)))$.
For this, we first express $K^\mathcal{H}$ in terms of the coordinates $q\in P$ and $s\in S \oplus jS$.
For $s\in S\oplus jS$, let $s^{x}$ be the projection of $s$ on $S^{x}\oplus jS^{x}$ given by the decomposition \eqref{eq:221} with $S$ replaced by $S^{x}$.
Let us define $p^{x}:S\oplus jS\rightarrow V$ by $p^{x} s:=s^{x}$ and a self-adjoint operator $A^{x}$ on $(S\oplus jS, h)$ by
\begin{align*}
A^{x}:=2(p^{x})^*R_{x}p^{x}.
\end{align*}
Then we can write
\begin{align*}
\langle{x},Q(s^{x},\overline{s^{x}}^{x})\rangle_U=h(s,A^{x}\overline{s}^S).
\end{align*}
For $0\leq k\leq \dim_\mathbb{C}P$, let
\begin{align*}
\Lambda_k:=\left\{{x}\in U\Bigm| \begin{array}{c}g_{x}|_{P\times P}\text{ is positive semi-definite},\\ \dim_\mathbb{C}N_{x}=k,\quad N_{x}\subset \ker(g_{x})\end{array}\right\}
\end{align*}
and put $\Lambda:=\coprod_{k=0}^{\dim_\mathbb{C}P}\Lambda_k$ and for $({x},\chi)\in \Lambda\times S^*$, let
\begin{equation}\label{eq:Ax}\begin{split}
L^{{x},\chi}(z,v,z',v'):&=e^{i\langle {x},z-\overline{z'}-2iQ(q,q')\rangle_U}e^{h(s,A^{x}\overline{s}^S)}e^{h(\overline{s'}^S,A^{x}s')}e^{-i\langle\chi,s-\overline{s'}\rangle}
\\&=e^{i\langle {x},z-\overline{z'}-2iQ(v,v')\rangle_U}e^{h(s-\overline{s'}^S,A^{x}(\overline{s}^S-s'))}e^{-i\langle\chi,s-\overline{s'}\rangle}
\end{split}\end{equation}
with $v=q+s,v'=q'+s'\,(q,q'\in{P},s,s'\in S\oplus jS)$.
\begin{corollary}\label{cor:115}
Under the condition \eqref{condi:key}, for any $G^W$-invariant Hilbert subspace $\mathcal{H}$ of $\mathcal{O}({\mathcal{S}(\Omega,Q)})$, there exists unique Radon measures $m_{k}$ on $\Lambda_k\times S^*\,(k=0,1,\cdots,\dim_\mathbb{C}P)$ such that the reproducing kernel $K^\mathcal{H}$ of $\mathcal{H}$ is expressed as
\begin{align*}
K^\mathcal{H}(z,v,z',v')=\sum_{k=0}^{\dim_\mathbb{C}P}\int_{\Lambda_k\times S^*}L^{{x},\chi}(z,v,z',v')\,dm_k({x},\chi).
\end{align*}
\end{corollary}
\begin{proof}
By Theorem \ref{th:uniquemeasure}, it is enough to show that the map
\begin{align*}
\Lambda\times S^*\ni({x},\chi)\mapsto L^{{x},\chi}(z,v,z',v')\in\ext (\Gamma_{G^W}(\mathcal{S}(\Omega,Q)))
\end{align*}
is an admissible parametrization of $\ext (\Gamma_{G^W}(\mathcal{S}(\Omega,Q)))$ when we equip $\Lambda\simeq \coprod_{k=0}^{\dim_\mathbb{C} P}\Lambda_k$ with the product topology.
In the following, when we refer the topology of $\Lambda$, it is always assumed to be the initial topology unless stated otherwise.
For $({x}',\chi')\in\Lambda_k\times S^*$, take an orthonormal basis
\begin{align*}
(q_1,q_2,\cdots, q_{\dim_\mathbb{C}P-k})
\end{align*}
of $P^{{x}'}$ with respect to the Hermitian inner product whose real part is $g_{{x}'}$.
For ${x}$ within a suitable neighborhood of ${x}'$, we produce, by the Gram-Schmidt process, the orthonormal vectors $(q_1({x}),q_2({x}),\cdots, q_{\dim_\mathbb{C} P-k}({x}))$ in $P$ with respect to the Hermitian form $h_{x}$ such that $\re h_{x}=g_{x}$.
Then the maps $q_l(\cdot)\,(1\leq l\leq \dim_\mathbb{C}P-k)$ are continuous at ${x}={x}'$, which implies that $\Lambda_k\subset \Lambda_{\geq k}:=\coprod_{l=k}^{\dim_\mathbb{C}P}\Lambda_l$ is open.
For ${x}\in\Lambda_{\geq k}$ within a suitable neighborhood of ${x}'$, the projection of $v\in V$ on $P^{\perp,\,g_{x}}=S^{x}\oplus jS^{x}\oplus N_{x}$ along with the decomposition $V=P^{\perp,\,g_{x}}\oplus P^{x}$ is expressed, up to a vector in $N_{x}$, as
\begin{align*}
v-\sum_{l=1}^{\dim_\mathbb{C}P-k}h_{x}(v,q_l({x}))q_l({x}),
\end{align*}
and hence as $F(v,{x},q_1({x}),q_2({x}),\cdots,q_{\dim_\mathbb{C}P-k}({x}))$ for some polynomial map $F$.
So, the map
\begin{align*}
\Lambda\times V\times V\ni({x},v,v')\mapsto \langle{x},Q(s^{x}-\overline{s'}^{x},\overline{s}^{x}-{s'}^{x})\rangle_U\in\mathbb{C}
\end{align*}
is locally of the form $F'(v,\overline{v'}^{W_0},{x},q_1({x}),q_2({x}),\cdots,q_{\dim_\mathbb{C}P-k}({x}))$ for some polynomial map $F'$.
Therefore, the map $\Lambda_{\geq k}\times S^*\ni({x},\chi)\mapsto L^{{x},\chi}(z,v,z',v')\in\mathcal{O}(\mathcal{S}(\Omega,Q))$ is continuous at $({x},\chi)=({x}',\chi')$.
To complete the proof, it is enough to show the space $\Lambda\simeq\coprod_{k=0}^{\dim_\mathbb{C}P}\Lambda_k$ with the product topology is second countable and locally compact.
The latter follows from the fact that $\Lambda_k\subset \Lambda_{\geq k}$ is open for each $0\leq k\leq\dim_\mathbb{C}P$.
\end{proof}

\subsection{Construction of intertwining operators}\label{subsect:intertwiningoperator}
In this subsection, we show \ref{condi:mainreal} $\Rightarrow$ \ref{condi:mainmf} of Theorem \ref{th:main}.
The condition \ref{condi:mainmf} can be expressed as $\ext(\Gamma_{G^W}(\mathcal{S}(\Omega,Q)))/\mathbb{R}^\times\hookrightarrow\widehat{G^W}$.
The set on the left hand side is determined in Theorem \ref{th:4.4} in Sect. \ref{subsect:extremal}.
To understand the unitary equivalences, it is natural to employ the orbit method and construct irreducible unitary representations from coadjoint orbits.
In this context, following the Auslander-Kostant theory \cite{auslander_polarization_1971}, we utilize the holomorphically induced representation  defined by a complex polarization, rather than a real polarization.
Proposition \ref{prop:245} provides a realization of the holomorphically induced representation on the Fock space $\mathcal{F}_x\subset\mathcal{O}(P)$ defined below, which can be seen as a generalization of the Bargmann-Fock representation of the Heisenberg group.
We complete the proof by constructing an intertwining operator from $\mathcal{F}_x$ to $\mathcal{O}(\mathcal{S}(\Omega,Q))$ and observing the coadjoint orbits.

For a finite-dimensional vector space $V_0$ over $\mathbb{R}$, let us consider the pushforward measure $\mu_{V_0}$ of the Lebesgue measure by a linear isomorphism of $\mathbb{R}^{\dim V_0}$ onto $V_0$.
Suppose that $\mu^{x}:=\mu_{{P}^{x}}$ is normalized so that
\begin{align*}
\int_{{P}^{x}}e^{-2\langle {x},Q(q,q)\rangle_U}d\mu^{x}(q)=1.
\end{align*}
Put
\begin{align*}
\mathcal{F}_{x}:=\left\{F\in\mathcal{O}({P})\Bigm| \begin{array}{c}F(q_1+q_2)=F(q_1)\text{ for all }q_1\in{P}\text{ and }q_2\in N_{x},\\
\int_{{P}^{x}}|F(q)|^2e^{-2\langle {x},Q(q,q)\rangle_U}\,d\mu^{x}(q)<\infty\end{array}\right\}.
\end{align*}
By \eqref{eq:224}, we can define for $\sigma\in (S_{x})^*$, a unitary representation $(V_{{x},\sigma},\mathcal{F}_{x})$ by
\begin{align*}
\begin{split}
V_{{x},\sigma}({\bf n}(x_0,q_0+s_0))F(q):=e^{-i\langle{x},x_0+2iQ(q,q_0)-iQ(q_0,q_0)\rangle_U}e^{i\langle\sigma,s_0\rangle}F(q-q_0)\\
(x_0\in U,q_0\in{P}^{x},s_0\in S_{x},F\in\mathcal{F}_x).
\end{split}
\end{align*}
Let $X_{{x},\sigma}\in(\mathfrak{g}^W)^*$ be given by
\begin{align*}
X_{{x},\sigma}(x_0,q_0+s_0):=-\langle{x},x_0\rangle_U+\langle\sigma,s_0\rangle\quad(x_0\in U,q_0\in{P}^{x},s_0\in S_{x}).
\end{align*}
Let
\begin{align*}
\mathfrak{p}:=U_\mathbb{C}\oplus(S_{x})_\mathbb{C}\oplus\{q+ijq\mid q\in{P}^{x}\},\quad \mathfrak{d}:=\mathfrak{p}\cap\mathfrak{g}^W,\quad D:=G^{S_{x}}.
\end{align*}
\begin{proposition}\label{lem:244}
The complex subalgebra $\mathfrak{p}$ is a positive polarization at $X_{{x},\sigma}\in(\mathfrak{g}^W)^*$ and satisfies the Pukanszky condition.
\end{proposition}
\begin{proof}
The isotropicity of $\mathfrak{p}$ follows from Proposition \ref{prop:109}.
Also, recalling $\mathfrak{g}_1\oplus S_{x}=\mathfrak{g}^W({x})$ and \eqref{eq:224}, we see the maximality of $\mathfrak{p}$.
The positivity of $\mathfrak{p}$ follows from Proposition \ref{prop:109} and \eqref{positivity238}.
\end{proof}

Let $\mathcal{H}(X_{{x},\sigma},\mathfrak{p},G^W)$ be the space of smooth functions $\phi$ on $G^W$ satisfying
\begin{align}
\phi(g\exp b)&=e^{-\langle X_{{x},\sigma},b\rangle}\phi(g)\quad(g\in G^W,b\in\mathfrak{d})\label{eq:411},\\
&\int_{G^W/D}|\phi|^2\,dm_{G^W/D}<\infty,\label{eq:412}\\
dR(a)\phi&=-i\langle X_{{x},\sigma},a\rangle\phi\quad(a\in\mathfrak{p}),\label{eq:413}
\end{align}
where $m_{G^W/D}$ denotes a nonzero $G^W$-invariant measure on $G^W/D$, and for $a_1, a_2\in\mathfrak{g}^W$, we define $dR(a_1+ia_2)\phi\in C^\infty(G^W)$ by
\begin{align*}
dR(a_1+ia_2)\phi(g):=\frac{d}{dt}\Bigr|_{t=0}\phi(ge^{ta_1})+i\frac{d}{dt}\Bigr|_{t=0}\phi(ge^{ta_2})\quad(g\in G^W).
\end{align*}
The holomorphically induced representation $\rho=\rho(X_{{x},\sigma},\mathfrak{p},G^W)$ is given by
\begin{align*}
\rho(g)\phi(g'):=\phi(g^{-1}g')\quad(\phi\in\mathcal{H}(X_{{x},\sigma},\mathfrak{p},G^W),g,g'\in G^W).
\end{align*}
For $v\in{P}$, let $v^{x}\in{P}^{x}$ be the orthogonal projection of $v$ on ${P}^{x}$ with respect to $\re h$.
Let $\Psi_{x}:\mathcal{H}(X_{{x},\sigma},\mathfrak{p},G^W)\rightarrow C^\infty({P})$ be given by
\begin{align*}
\Psi_{x}\phi(q)=e^{\langle{x},Q(q,q)\rangle_U}\phi({\bf n}(0,q^{x}))\quad(q\in{P}).
\end{align*}
\begin{proposition}\label{prop:245}
The following hold.
\begin{enumerate}[label*=\textup{(\arabic*)}]
\item
The map $\Psi_{x}$ gives a $G^W$-intertwining operator from $\mathcal{H}(X_{{x},\sigma},\mathfrak{p},G^W)$ onto $\mathcal{F}_{x}$.
\item
$\mathcal{H}(X_{{x},\sigma},\mathfrak{p},G^W)\neq \{0\}$.
\end{enumerate}
\end{proposition}
\begin{proof}
(1)
The conditions \eqref{eq:411} and \eqref{eq:413} imply that $\Psi_{x}\phi\in\mathcal{O}({P})$, and \eqref{eq:412} shows that $\Psi_{x}$ is an isometry onto $\mathcal{F}_{x}$ up to a scalar multiplication.
For $x_0\in U$, $s_0\in S_{x}$, $q_0\in{P}^{x}$, we see from \eqref{eq:411} that
\begin{align*}
&\Psi_{x}\rho({\bf n}(x_0,s_0+q_0))\phi(q)
\\&=\phi({\bf n}(0,q^{x}-q_0){\bf n}(-2\im Q(q^{x}-q_0,-s_0)-x_0-2\im Q(s_0+q_0,q^{x}),-s_0))
\\&\quad\cdot e^{\langle{x},Q(q,q)\rangle_U}
\\&=e^{-i\langle{x},x_0+2\im Q(q_0,q^{x})+iQ(q,q)\rangle_U}e^{-i\langle\sigma,-s_0\rangle}\phi({\bf n}(0,q^{x}-q_0))
\\&=e^{-i\langle{x},x_0\rangle_U}e^{2\langle{x},Q(q,q_0)\rangle_U}e^{-\langle {x},Q(q_0,q_0)\rangle_U}e^{i\langle\sigma,s_0\rangle}\Psi_{x}\phi(q-q_0)
\\&=V_{{x},\sigma}({\bf n}(x_0,q_0+s_0))\Psi_{x}\phi(q),
\end{align*}
and we are led to the conclusion.

(2) Since $\mathcal{F}_{x}\neq \{0\}$, the assertion follows from (1).
\end{proof}
By Fujiwara \cite{fujiwara_unitary_1974}, Propositions \ref{lem:244} and \ref{prop:245}(2) imply the following.
\begin{proposition}\label{th:246}
$\rho$ is irreducible and the orbit $G^WX_{{x},\sigma}$ is mapped by the Kirillov-Bernat correspondence to the unitary equivalence class of $\rho$.
\end{proposition}
Let us denote by $\Ad ^*$ the coadjoint representation of $G^W$.
\begin{remark}\label{rem:226}
We can see from Proposition \ref{prop:109} that
\begin{align*}
\langle \mathrm{Ad}^*({\bf n}(x,v))X_{{x},\sigma},s\rangle=\langle\sigma,s\rangle
\end{align*}
for $x\in U$, $v\in W$, and $s\in S_{x}$.
Hence Propositions \ref{prop:245}(1) and \ref{th:246} show that the following conditions for $\sigma,\sigma'\in(S_{x})^*$ are equivalent:
\begin{enumerate}[label=(\roman*)]
\item
$V_{{x},\sigma}\simeq V_{{x},\sigma'}$\quad(as unitary representations of $G^W$);
\item
$\sigma=\sigma'$.
\end{enumerate}
\end{remark}
Noting Remark \ref{decomposition242} and \eqref{eq:224}, for $F\in\mathcal{F}_{x}$, let $\Phi_{{x},\sigma}F$ be the function on ${\mathcal{S}(\Omega,Q)}$ defined by
\begin{align*}
\Phi_{{x},\sigma}F(z,q+s):=e^{\langle{x},iz+Q(s,\overline{s}^{x})\rangle_U}e^{-i\langle \sigma,s\rangle}F(q)\quad(q\in{P}, s\in S^{x}\oplus jS^{x}).
\end{align*}
\begin{proposition}\label{prop:4.5}
When $\sigma|_{N_{x}}=0$, the operator $\Phi_{{x},\sigma}$ intertwines $V_{{x},\sigma}$ with $\pi_0$.
\end{proposition}
\begin{proof}
For $F\in\mathcal{F}_{{x},\sigma}$, $x_0\in U$, $q_0\in{P}$, $s_0\in S^{x}$, we have
\begin{align*}
&\pi_0(\exp x_0)\Phi_{{x},\sigma}F(z,q+s)
\\&=\Phi_{{x},\sigma}F(z-x_0,q+s)
\\&=e^{-i\langle \sigma,s\rangle}e^{\langle {x},iz-ix_0+Q(s,\overline{s}^{x})\rangle_U}F(q)=\Phi_{{x},\sigma}V_{{x},\sigma}(\exp x_0)F(z,q+s)
\end{align*}
and
\begin{align*}
&\pi_0(\exp(q_0+s_0))\Phi_{{x},\sigma}F(z,q+s)
\\&=\Phi_{{x},\sigma}F({\bf n}(0,-q_0-s_0)(z,q+s))
\\&=e^{\langle{x},iz+Q(s,\overline{s}^{x})\rangle_U}e^{-i\langle\sigma,s\rangle}e^{i\langle\sigma,s_0\rangle}e^{-\langle{x},Q(q_0,q_0)\rangle_U}e^{2\langle{x},Q(q,q_0)\rangle_U}F(q-q_0),
\end{align*}
and $\sigma|_{N_{x}}=0$ implies that this equals
\begin{align*}
e^{\langle{x},iz+Q(s,\overline{s}^{x})\rangle_U}e^{-i\langle\sigma,s\rangle}V_{{x},\sigma}(\exp(q_0+s_0))F(q)=\Phi_{{x},\sigma}V_{{x},\sigma}(\exp(q_0+s_0))F(z,q+s).
\end{align*}
\end{proof}
\begin{remark}\label{rem:117}
When $\sigma|_{N_{x}}=0$, the map $\Phi_{{x},\sigma}$ does not depend on the choice of $S^{x}$.
\end{remark}
We can define a natural Hilbert space structure on $\Phi_{{x},\sigma}(\mathcal{F}_{x})$.
Let us denote by $K^{{x},\sigma}\in\Gamma_{G^W}({\mathcal{S}(\Omega,Q)})$ its reproducing kernel.
\begin{proposition}\label{prop:114}
One has
\begin{equation}\begin{split}\label{eq:248}
&K^{{x},\sigma}(z,q+s,z',q'+s')
\\&=\exp({i\langle{x},z-\overline{z'}-2iQ(q,q')-iQ(s,\overline{s}^x)-iQ(\overline{s'}^x,s')\rangle_U})
\\&\quad\cdot e^{-i\langle\sigma, s-\overline{s'}^{x}\rangle}
\end{split}
\end{equation}
with $q,q'\in{P}, s,s'\in S^{x}\oplus jS^{x}$.
\end{proposition}
\begin{proof}
Fixing $(z',q'+s')\in{\mathcal{S}(\Omega,Q)}$ with $q'\in{P}$ and $s'\in S^{x}$, define
\begin{align*}
f:=e^{\langle{x},-i\overline{z'}+Q(\overline{s'}^{x},s')\rangle_U} e^{i\langle \sigma,\overline{s'}^{x}\rangle}\Phi_{{x},\sigma}(K_{q'}^{\mathcal{F}_{x}})\in\mathcal{O}(\mathcal{S}(\Omega,Q)).
\end{align*}
Then for $F\in\mathcal{F}_{x}$, we have
\begin{align*}
(\Phi_{{x},\sigma}(F),f)_{\Phi_{{x},\sigma}(\mathcal{F}_{x})}&=(F,e^{\langle{x},-i\overline{z'}+Q(\overline{s'}^{x},s')\rangle_U}e^{i\langle\sigma,\overline{s'}^{x}\rangle}K_{q'}^{\mathcal{F}_{x}})_{\mathcal{F}_{x}}
\\&=e^{\langle{x},iz'+Q(s',\overline{s'}^{x})\rangle_U}e^{-i\langle\sigma,s'\rangle}F(q')
\\&=\Phi_{{x},\sigma}F(z',q'+s'),
\end{align*}
which shows that $K^{{x},\sigma}(z,v,z',v')=f(z,v)$.
Now, we get the desired expression from $K^{\mathcal{F}_{x}}(q,q')=e^{2\langle{x},Q(q,q')\rangle_U}\,(q,q'\in{P})$.
\end{proof}
Comparing \eqref{eq:249} and \eqref{eq:248}, we conclude that the representation $(\pi_0,\mathcal{H})$ corresponds to the coadjoint orbit $G^WX_{{x},\sigma}$ by the Kirillov-Bernat map with ${x}={\nu}$ and $\sigma$ the zero extension of $\chi$ along with the decomposition
\begin{align*}
S_{\nu}=S^{\nu}\oplus N_{\nu}.
\end{align*}
Taking into account Remarks \ref{rem:226} and \ref{rem:117}, the following is just a corollary of Theorem \ref{th:4.4} and Proposition \ref{prop:114}.
\begin{corollary}\label{cor:112}
Suppose that a real subspace $W\subset V$ satisfies $\im Q(S,S)=\{0\}$.
Then the representation $\pi_0$ of $G^W$ is multiplicity-free.
\end{corollary}

\subsection{Specific form of an admissible parametrization}\label{subsect:specificform}
In this subsection, we consider an example of $W$ constructed from projections defined by a Jordan frame and the Jordan algebra representation.
We present an explicit formula for $h(s, A^{x} \overline{s}^S) \,(s\in S, x\in \Lambda)$, which is part of the description of our admissible parametrization in \eqref{eq:Ax}.
We observe that as a function of $S\oplus jS$, it reduces to a scalar multiplication of $h(s, \overline{s}^S)$.

Let $e_1, e_2,\cdots, e_r$ be a Jordan frame of $U$.
For the sake of simplicity, we assume that $R_x=0$ iff $x=0$ for $x\in U$ and $\langle e_1,e_1\rangle_U=1$.
For $x\in U$, the Peirce decomposition allows us to write $x=x_1+x_{1/2}+x_0$ with $x_\lambda\in U(e_1,\lambda):=\{u\in U\mid e_1u=\lambda u\}\,(\lambda=0,1/2,1)$.
For $v,v'\in V$, we see that
\begin{align*}
\langle x,Q(R_{e_1}v,R_{e_1}v')\rangle_U=2h(v,R_{P(e_1)x}v')=2h(v,R_{x_1}v')=\langle x_1,e_1\rangle_U\langle{e_1},Q(v,v')\rangle_U.
\end{align*}
Hence we can find a real subspace $S\subset R_{e_1}V$ such that $\im Q(S,S)=\{0\}$ and $S\oplus jS=R_{e_1}V$.
Let $e':=e_2+e_3+\cdots+e_r$, $P:=R_{e'}V$.
For $u,u'\in U$, put $\boxx{u}{u'}:=T_{uu'}+[T_u,T_{u'}]\in\mathfrak{g}(\Omega)$.
\begin{lemma}[c.f. {\cite[Lemma VI.3.1]{faraut_analysis_1994}}]\label{lem:adjointofFK}
For $y\in U(e_1,1/2)$, one has
\begin{align*}
\exp{}^t(\boxx{2y}{e_1})(x_1+x_0)=\left(x_1+2T_{e_1}(T_y)^2x_0\right)+2T_yx_0+x_0.
\end{align*}
\end{lemma}
\begin{proof}
For $y\in U(e_1,1/2)$, we have
\begin{align*}
{}^t(\boxx{2y}{e_1})x_1&=0,\\
{}^t(\boxx{2y}{e_1})x_{1/2}&=2T_{e_1}T_yx_{1/2},\\
{}^t(\boxx{2y}{e_1})x_0&=2T_yx_0,
\end{align*}
which leads to the desired equality.
\end{proof}
\begin{lemma}\label{lem:nilpotent}
The following hold.
\begin{enumerate}[label*=\textup{(\arabic*)}]
\item
$e^{\beta(\boxx{2y}{e_1})}|_P=\Id_P$.\label{nilpotent1}
\item
$e^{\beta(\boxx{2y}{e_1})^*}|_{S\oplus jS}=\Id_{S\oplus jS}$.\label{nilpotent2}
\item
$R_{e_1}e^{\beta(\boxx{2y}{e_1})}=R_{e_1}$.\label{nilpotent3}
\end{enumerate}
\end{lemma}
\begin{proof}
We have $\beta(\boxx{2y}{e_1})R_{e'}=2R_{e'}R_{y}R_{e'}=0$, which implies \ref{nilpotent1}.
On the other hand, we see that $\beta(\boxx{2y}{e_1})^*R_{e_1}=2R_{e_1}R_yR_{e_1}=0$, which implies \ref{nilpotent2} and by taking the adjoint, we see \ref{nilpotent3}.
\end{proof}
\begin{proposition}\label{prop:factor}
Suppose that $N_x\subset \ker(g_x)$.
Then there exists $y\in U(e_1,1/2)$ such that $x_{1/2}=2T_yx_0$.
\end{proposition}
\begin{proof}
First, for $v\in V$, we have the following equivalences:
\begin{align*}
R_{e'}v\in N_x\Leftrightarrow h(v,R_{e'}R_xR_{e'}V)=\{0\}\Leftrightarrow h(v,R_{x_0}V)=\{0\},
\end{align*}
\begin{align*}
R_{e'}v\in \ker(g_x)\Leftrightarrow h(v,R_{e'}R_xV)=\{0\}\Leftrightarrow h(v,R_{e'}(R_{x_{1/2}}+R_{x_0})V)=\{0\}.
\end{align*}
Hence we see from $N_x\subset \ker(g_x)$ that
\begin{align*}
(R_{x_0}V)^{\perp,\,\re h}\subset(R_{e'}R_{x_{1/2}}V)^{\perp,\,\re h},
\end{align*}
and hence $R_{e'}R_{x_{1/2}}V\subset R_{x_0}V$.
Then we can find $B\in\mathfrak{gl}(V)$ such that $R_{e'}R_{x_{1/2}}=R_{x_0}B$.
Thus we get
\begin{align}\label{eq:factor1}
R_{x_{1/2}}=2(R_{x_0}B+B^*R_{x_0}).
\end{align}
So, if $u\in U$ satisfies
\begin{align}\label{eq:factor3}
\boxx{x_0}{u}=0,\quad T_ux_0=0,
\end{align}
and then we have $R_{x_0}R_{u}=R_uR_{x_0}=0$, and it follows from \eqref{eq:factor1} that
\begin{align}\label{eq:factor2}
R_uR_{x_{1/2}}R_u=0.
\end{align}

Next, noting that $U(e_1,0)$ is a subalgebra of $U$, let $x_0=\sum_{k=2}^r\lambda_k{e_k}'\,(\lambda_k\neq 0\text{ iff }2\leq k\leq r_0)$ be the decomposition in Theorem \ref{th:228}.
For $u={e_k}'$, with $k>r_0$, we have \eqref{eq:factor3} and hence $R_{P({e_k}')x_{1/2}}=0$ by \eqref{eq:factor2}.
By our assumption, we obtain $P({e_k}')x_{1/2}=0$, which shows the assertion.
\end{proof}
\begin{proof}[Proof of Theorem \ref{th:main2}]
Let $x\in U$ satisfy the assumption of Proposition \ref{prop:factor} and let $u:=(x_1-2T_{e_1}{(T_y)}^2x_0)+x_0$.
Then we have for $s\in S\oplus jS$,
\begin{equation}\begin{split}\label{eq:findingscalar}
h(s,A^{x}\overline{s}^S)
&=\langle x,Q(p^xs,\overline{p^xs}^x)\rangle_U
\\&=2h(R_xp^xs,\overline{p^xs}^x)
\\&=2h(R_xp^xs,\overline{s})=2h(R_{e^{{}^t(\boxx{2y}{e_1})}u}p^xs,\overline{s}),
\end{split}\end{equation}
where in the third and last equalities, we have used the fact $R_{x}p^xs\in S\oplus jS$ and Lemma \ref{lem:adjointofFK}, respectively.
By Lemma \ref{lem:105}, we have
\begin{align*}
R_{e^{{}^t(\boxx{2y}{e_1})}u}=e^{\beta(\boxx{2y}{e_1})^*}R_{u}e^{\beta(\boxx{2y}{e_1})}.
\end{align*}
On the other hand, by Lemma \ref{lem:nilpotent}\ref{nilpotent2}, we have
\begin{align*}
e^{\beta(\boxx{2y}{e_1})^*}R_{u}e^{\beta(\boxx{2y}{e_1})}p^xs=R_{u}e^{\beta(\boxx{2y}{e_1})}p^xs\in S\oplus jS.
\end{align*}
Furthermore, we can write
\begin{align*}
p^xs=s+q,\quad e^{\beta(\boxx{2y}{e_1})}(s+q)=s+q'
\end{align*}
for some $q,q'\in P$ by Lemma \ref{lem:nilpotent}\ref{nilpotent1}, \ref{nilpotent3}.
Therefore, the last expression of \eqref{eq:findingscalar} equals
\begin{align*}
2h(R_{u}{(s+q')},\overline{s})&=h(2R_{x_1-2T_{e_1}{(T_y)}^2x_0}s,\overline{s})
\\&=\langle x_1-2T_{e_1}{(T_y)}^2x_0,e_1\rangle_Uh(s,\overline{s}).
\end{align*}
\end{proof}

\section{Multiplicity-free unitary representation on the Bergman space}\label{sect:L2}
In this section, we show \ref{condi:mainl2}$\Rightarrow$\ref{condi:mainreal} of Theorem \ref{th:main}.
We use the multiplicity-free direct integral decomposition of the unitary representation of $G^V$ on the space of all $L^2$ holomorphic functions.
Additionally, in Proposition \ref{prop:5.2}, we provide a description of the restrictions of the irreducible representations $V_{x}=V_{x,0}\,(x \in U)$ of $G^V$ to $G^W$, which serves as a crucial component in our proof.
The condition \ref{condi:mainl2} ensures that for each $x \in \Omega$, the representation $V_{x}|_{G^W}$ is multiplicity-free, leading to \ref{condi:mainreal} as demonstrated in Theorem \ref{th:220}.

Let $W\subset V$ be a real subspace.
We do not impose on $W$ any other conditions.
We assume that $\mu_U$ stands for the pushforward measure of the Lebesgue measure by an isometry from the space $\mathbb{R}^N$ with the standard inner product onto $(U,\langle\cdot,\cdot\rangle_U)$.
We denote by $\mu$ the natural complete measure on $U_\mathbb{C}\oplus V$ induced by $\mu_U$ and $\mu_V$.
Let
\begin{align*}
L_a^2({\mathcal{S}(\Omega,Q)}):=L^2({\mathcal{S}(\Omega,Q)},\mu)\cap\mathcal{O}({\mathcal{S}(\Omega,Q)}).
\end{align*}
We see an integral expression for the Bergman kernel of ${\mathcal{S}(\Omega,Q)}$.
For $u\in\Omega$, let
\begin{align*}
I(u):=\int_\Omega e^{-2\langle u,y\rangle_U}\,d\mu_U(y),\quad I_Q(u):=\int_V e^{-2\langle u,Q(v,v)\rangle_U}\,d\mu_V(v).
\end{align*}
\begin{theorem}[\cite{gindikin_analysis_1964}]\label{th:5.3}
For $(z,v),(z',v')\in{\mathcal{S}(\Omega,Q)}$, the reproducing kernel $K$ of $L_a^2({\mathcal{S}(\Omega,Q)})$ is given by
\begin{align*}
K(z,v,z',v')=\frac{1}{(2\pi)^N}\int_\Omega e^{i\langle u,z-\overline{z'}-2iQ(v,v')\rangle_U}I(u)^{-1}I_Q(u)^{-1}\,d\mu_U(u).
\end{align*}
\end{theorem}
In the special case that $W=V$, Propositions \ref{prop:4.5} and \ref{prop:114} imply that for $x\in U$ satisfying \eqref{positivity238}, the representation $V_{x}:=V_{{x},0}$ of $G^V$ can be realized in $\mathcal{O}({\mathcal{S}(\Omega,Q)})$, and the corresponding reproducing kernel is given by
\begin{align*}
e^{i\langle{x},z-\overline{z'}-2iQ(v,v')\rangle_U}
\end{align*}
up to a constant.
Hence, in view of Theorem \ref{th:5.3}, it follows that
\begin{equation}\label{eq:219}
L_a^2({\mathcal{S}(\Omega,Q)})\simeq\int_\Omega^\oplus V_{u}\,d\mu_U(u).
\end{equation}
Indeed, the map
\begin{align*}
&\int_\Omega^\oplus\mathcal{F}_{u}I(u)^{-1}I_Q(u)^{-1}\,d\mu_U(u)\ni f\\\quad&\mapsto\int_\Omega\Phi_{{x},0}f(u)(\cdot)I(u)^{-1}I_Q(u)^{-1}\,d\mu_U(u)\in\mathcal{O}(\mathcal{S}(\Omega,Q))
\end{align*}
is continuous, and
the kernel $\mathcal{K}$ of the map is given by
\begin{align*}
\mathcal{K}=\int_{\Omega_0}^\oplus \mathcal{F}_{u}I(u)^{-1}I_Q(u)^{-1}\,d\mu_U(u)
\end{align*}
for some measurable set $\Omega_0\subset\Omega$ due to \cite[Theorem 1.2]{mautner_unitary_1950}.
On the other hand, the Hilbert space
\begin{align*}
\int_\Omega^\oplus\mathcal{F}_{u}I(u)^{-1}I_Q(u)^{-1}\,d\mu_U(u)/\mathcal{K}
\end{align*}
is a $G^V$-invariant Hilbert subspace of $\mathcal{O}(\mathcal{S}(\Omega,Q))$ and has the same reproducing kernel as $L_a^2(\mathcal{S}(\Omega,Q))$, and hence isomorphic to it.
Let us show that $\mathcal{K}=\{0\}$.
For ${\bf z}\in\mathcal{S}(\Omega,Q)$, we have
\begin{align*}
&\int_{\Omega_0}(K_{\bf z}^{{u},0},K_{\bf z}^{{u},0})_{\Phi_{{u},0}(\mathcal{F}_{u})}I(u)^{-1}I_Q(u)^{-1}\,d\mu_U(u)
\\&=\int_{\Omega_0}K_{\bf z}^{{u},0}({\bf z})I(u)^{-1}I_Q(u)^{-1}\,d\mu_U(u)=0,
\end{align*}
which implies that $\mu_U(\Omega_0)=0$.
Suppose that ${x}\in U$ satisfies \eqref{positivity238} with ${P}$ replaced by $V$.
For $v\in V$, let ${x}_v\in {(V_\mathbb{R})}^*$ be given by
\begin{align*}
\langle {x}_v,v'\rangle:=\langle {x},[v,v']\rangle_U\quad(v'\in V).
\end{align*}
Taking a real subspace $S^{x}\subset V$ complementary to $W^{\perp,\,g_{x}}+jW$, we put $\mu_{x}:=\mu_{jS^{x}}$.
Let $p:(\mathfrak{g}^V)^*\rightarrow(\mathfrak{g}^W)^*$ be the canonical projection, and put
\begin{align*}
p_{x}(v):=\widehat{\rho_{G^W}}\circ p(-{x}+{x}_v)\quad(v\in V),
\end{align*}
where we regard $x$ as a vector in $U^*\subset(\mathfrak{g}^V)^*$ by means of $\langle\cdot,\cdot\rangle_U$.
Then we have the following proposition.
\begin{proposition}\label{prop:5.2}
We have
\begin{equation}\label{eq:227}
V_{x}|_{G^W}\simeq \int_{jS^{x}}^\oplus n(p_{x}(v))p_{x}(v)\,d\mu_{x}(v)
\end{equation}
with $n(p_{x}(v))\equiv 1\text{ or }\infty$.
More precisely, the following are equivalent:
\begin{enumerate}[label=\textup{(\roman*)}]
\item
$n(p_{x}(v))=1$ for all $v\in jS^{x}$;
\item[\textup{(i')}]
$V_{x}$ is multiplicity-free as a unitary representation of $G^W$;
\item
For any $v\in jW^{\perp,\, g_{x}}$, there exists $w\in W$ such that $v+w\in N_{x}$.\label{cond:iii}
\end{enumerate}
\end{proposition}
\begin{proof}
Since the details of the condition \ref{cond:iii} are dealt with in \cite[Proposition 5.3]{arashi_multiplicityfree_a}, we skip them here.
Let
\begin{align*}
V_{x}|_{G^W}\simeq\int_{\widehat{G^W}}n(\nu)\nu\,dm(\nu)
\end{align*}
be the disintegration of $V_{x}|_{G^W}$ in Theorem \ref{th:250}.
When $n(\nu)\geq 1$, we have $n(\nu)=1$ if and only if \ref{cond:iii} holds.
Note that the latter condition does not depend on $\nu\in\widehat{G^W}$, and hence is equivalent to $n(\nu)=1$ for all $\nu\in\widehat{G^W}$ with $n(\nu)\geq 1$.
Also we have $n(\nu)\in\{0,\infty\}$ if \ref{cond:iii} does not hold.
For a real subspace $W_0\subset V$, put
\begin{align*}
{x}_{W_0}:=\{{x}_v\mid v\in W_0\}\subset {(V_\mathbb{R})}^*.
\end{align*}
Taking a natural complete measure $m_1$ on $-{x}+{x}_V$, which is defined by finite measures on ${x}_{(jS^{x})}$, say $m_2$, and ${x}_{(jW^{\perp,\,g_{x}}+W)}$, equivalent to natural complete measures.
Let us consider
\begin{align*}
m:=(\widehat{\rho_{G^W}}\circ p)_*m_1.
\end{align*}
Then for an integrable function $f$, we have
\begin{align*}
\int_{\widehat{G^W}}f(\nu)\,dm(\nu)=\int_{-{x}+{x}_V}f(\widehat{\rho_{G^W}}(\nu|_{\mathfrak{g}^W}))\,dm_1(\nu),
\end{align*}
which equals
\begin{align*}
\int_{{x}_{(jS^{x})}}f(\widehat{\rho_{G^W}}(-{x}+\nu|_W))\,dm_2(\nu),
\end{align*}
and hence
\begin{align*}
\int_{jS^{x}}f(p_{x}(v))\frac{dm_2'}{d\mu_{x}}({x}_v)\,d\mu_{x}(v)
\end{align*}
up to a constant, where $m_2'$ is the pushforward measure of $m_2$ by the inverse mapping of $jS^{x}\ni v\mapsto {x}_v\in{x}_{(jS^{x})}$.
This gives the desired expression.
\end{proof}
Now we shall prove \ref{condi:mainl2} $\Rightarrow$ \ref{condi:mainreal} of Theorem \ref{th:main}.
\begin{theorem}\label{th:220}
If $(\pi_0,L_a^2({\mathcal{S}(\Omega,Q)}))$ is multiplicity-free as a unitary representation of $G^W$, then one has $\im Q(S,S)=\{0\}$.
\end{theorem}
\begin{proof}
Combining \eqref{eq:219} and \eqref{eq:227}, we have
\begin{equation}\begin{split}\label{eq:251}
(\pi_0, L_a^2({\mathcal{S}(\Omega,Q)}))&\simeq\int_\Omega^\oplus\int_{jS^{u}}^\oplus n(p_{u}(v))p_{u}(v)\,d\mu_{u}(v)\,d\mu_U(u)
\\&\simeq\int_{\Omega\times jS}^\oplus n(p_{u}(p^{u} v))p_{u}(p^{u} v)\,d\mu_{U\oplus V}(u,v).
\end{split}\end{equation}
Suppose that $(\pi_0, L_a^2({\mathcal{S}(\Omega,Q)}))$ is multiplicity-free.
Then we see from \eqref{eq:251} that
\begin{align*}
n(p_{u}(p^{u} v))=1,\quad \text{a.e. }(u,v)\in\Omega\times jS.
\end{align*}
and by Proposition \ref{prop:5.2}, this implies that
\begin{equation*}
jW^{\perp,\,g_x}\subset W,\quad \text{a.e. }x\in \Omega.
\end{equation*}
Hence
\begin{equation}
\langle {x},\im Q(S^{x},S^{x})\rangle_U=\{0\},\quad \text{a.e. }x\in \Omega,
\end{equation}
and this holds for all $x\in \Omega$ as a consequence of the continuity that we have seen in the proof of Corollary \ref{cor:115}.
Then by Lemma \ref{lem:111}, we have
\begin{align*}
j((W^{\perp,\,g_e})^{\perp,\,g_x})^{\perp,\,g_e}=jW^{\perp,\,g_{x^{-1}}}\subset W,
\end{align*}
since $\Omega$ is preserved under the mapping $U^\times\ni x\mapsto P(y)x^{-1}\in U\,(y\in\Omega)$.
Therefore, we have $jS\subset S^{\perp,\,g_x}\,(x\in\Omega)$, which implies that $\im Q(S,S)=\{0\}$.
\end{proof}

\section{Coisotropic action and the multiplicity-freeness property}\label{sect:coisotropic}
In this section, we prove \ref{condi:maincoiso} $\Rightarrow$ \ref{condi:mainreal} of Theorem \ref{th:main} in Theorem \ref{th:5.11}, and the converse in Theorem \ref{th:5.14}.
A central tool in these proofs is the pseudo-inverse map $\mathcal{I}_\Delta: \Omega \rightarrow U$, defined below.
In Proposition \ref{prop:5.13}, to show the coisotropicity of the group action, we explicitly determine the orthogonal complements of the tangent spaces of the group orbits with respect to the Bergman metric of $\mathcal{S}(\Omega, Q)$ over a certain submanifold of $\mathcal{S}(\Omega, Q)$.
Note that our proofs rely on a technical result, established in Proposition \ref{prop:comm}, which asserts that $T_{e'}$ with $e'=\mathcal{I}_\Delta(e)$ lies in the center of $\mathfrak{g}(\Omega)$.

For $y\in\Omega$, let $m_y$ denote the measure on $U$ given by
\begin{align*}
m_y:=e^{-2\langle \cdot,y\rangle_U}I^{-1}I_Q^{-1}\,\mu_U.
\end{align*}
Let $\Delta$ be the function on $\Omega$ defined by
\begin{align*}
\Delta(y):=\int_\Omega \,dm_y(u)\quad (y\in\Omega).
\end{align*}
Let $\mathcal{I}_\Delta:\Omega\rightarrow U$ be given by
\begin{align*}
\langle\mathcal{I}_\Delta(y),x\rangle_U=-\partial_x\log\Delta(y)=\frac{2}{\Delta(y)}\int_\Omega\langle u,x\rangle_U\,dm_y(u)\quad (x\in U).
\end{align*}
Let $e':=\mathcal{I}_\Delta(e)$.
The following lemma shows a basic property of the map $\mathcal{I}_\Delta$.
\begin{lemma}[{\cite[Lemma 2.5]{dorfmeister_homogeneous_1982}}]\label{lem:5.5}
$\mathcal{I}_\Delta$ defines a diffeomorphism from $\Omega$ to itself.
In addition, one has
\begin{align*}
\langle e',x\rangle_U=\partial_x\partial_e\log\Delta(e)\quad(x\in U).
\end{align*}
\end{lemma}
In the following lemma, we shall give formulae for the value $\tilde{g}_{{\bf z}}$ of the Bergman metric $\tilde{g}$ at ${\bf z}\in{\mathcal{S}(\Omega,Q)}$.
The tangent space $T_{{\bf z}}{\mathcal{S}(\Omega,Q)}$ will be naturally identified with $U_\mathbb{C}\oplus V$.
\begin{lemma}
For $\zeta,\zeta'\in U_\mathbb{C}$ and $\gamma,\gamma'\in V$, we have
\begin{equation}\label{eq:5.1}
\begin{split}
&\tilde{g}_{(z,v)}(\zeta,\zeta')
\\&=-\re (2\pi)^{-2N}K^{-2}(z,v,z,v)
\\&\quad\cdot\left\{(2\pi)^NK(z,v,z,v)\int_\Omega\langle u,\zeta\rangle_U\langle u,\overline{\zeta'}\rangle_U \,dm_{\im z-Q(v,v)}(u)\right.
\\&\quad\quad\left.-\left(\int_\Omega\langle u,\zeta\rangle_U \,dm_{\im z-Q(v,v)}(u)\right)
\left(\int_\Omega\langle u,\overline{\zeta'}\rangle_U \,dm_{\im z-Q(v,v)}(u)\right)\right\},
\end{split}
\end{equation}
\begin{equation}\label{eq:5.2}
\begin{split}
&\tilde{g}_{(z,v)}(\zeta,\gamma)
\\&=\re (2\pi)^{-2N}K^{-2}(z,v,z,v)
\\&\quad\cdot\left\{(2\pi)^NK(z,v,z,v)\int_\Omega 2i\langle u,\zeta\rangle_U\langle u,Q(v,\gamma)\rangle_U \,dm_{\im z-Q(v,v)}(u)\right.
\\&\quad\quad\left.-\int_\Omega i\langle u,\zeta\rangle_U \,dm_{\im z-Q(v,v)}(u)\int_\Omega 2\langle u,Q(v,\gamma)\rangle_U \,dm_{\im z-Q(v,v)}(u)\right\},
\end{split}
\end{equation}
\begin{equation}\label{eq:5.3}
\begin{split}
&\tilde{g}_{(z,v)}(\gamma,\gamma')
\\&=2\re (2\pi)^{-N}K^{-1}(z,v,z,v)
\\&\quad\cdot\left(2\int_\Omega\langle u,Q(v,\gamma')\rangle_U\langle u,Q(\gamma,v)\rangle_U \,dm_{\im z-Q(v,v)}(u)\right.
\\&\quad\quad\left.+\int_\Omega\langle u,Q(\gamma,\gamma')\rangle_U \,dm_{\im z-Q(v,v)}(u)\right)-4(2\pi)^{-2N}K^{-2}(z,v,z,v)
\\&\quad\cdot\int_\Omega\langle u,Q(v,\gamma')\rangle_U \,dm_{\im z-Q(v,v)}(u)\int_\Omega\langle u,Q(\gamma,v)\rangle_U \,dm_{\im z-Q(v,v)}(u).
\end{split}
\end{equation}
\end{lemma}
We shall show \ref{condi:maincoiso} $\Rightarrow$ \ref{condi:mainreal} of Theorem \ref{th:main} in Theorem \ref{th:5.11}.
\begin{theorem}\label{th:5.11}
Suppose that every $G^W$-orbit of ${\mathcal{S}(\Omega,Q)}$ is a coisotropic submanifold with respect to the Bergman metric, then we have $\im Q(S,S)=\{0\}$.
\end{theorem}
\begin{proof}
From the assumption, we see that
\begin{align*}
(T_{(iy,0)}G^W(iy,0))^{\perp,\,\tilde{g}_{(iy,0)}}\subset jT_{(iy,0)}G^W(iy,0)\quad (y\in\Omega)
\end{align*}
and
\begin{equation}\label{eq:255}
W^{\perp,\,\tilde{g}_{(iy,0)}|_{V\times V}}\subset jW\quad (y\in\Omega).
\end{equation}
Let $x\in\Omega$.
By Lemma \ref{lem:5.5}, there exists $y\in\Omega$ such that $x=\mathcal{I}_\Delta(y)$.
Then we see from \eqref{eq:5.3} that
\begin{equation}\label{eq:gxgtilde}
\begin{split}
g_x(v_1,v_2)&=\frac{2}{\Delta(y)}\int_\Omega\langle u,\re Q(v_1,v_2)\rangle_U\,dm_y(u)\\
\\&=\tilde{g}_{(iy,0)}(v_1,v_2)\quad(v_1,v_2\in V).
\end{split}
\end{equation}
Hence by \eqref{eq:255}, we have $jS\subset S^{\perp,\,g_{x^{-1}}}$.
This shows that
\begin{align*}
jS\subset S^{\perp,\,g_x}\quad(x\in\Omega),
\end{align*}
and hence $\im Q(S,S)=\{0\}$.
\end{proof}

Next, we prepare the necessary lemma and propositions to establish the converse of Theorem \ref{th:5.11}.
\begin{lemma}\label{lem:253}
One has
\begin{align*}
\partial_a\partial_b\log\Delta(e)=2(\tr T_{ab}+\tr R_{ab})\quad (a,b\in U).
\end{align*}
\end{lemma}
\begin{proof}
For $g=\exp(T_a)$, $x\in \Omega$, we have
\begin{align*}
I(gx)=e^{-\tr T_a}I(x),\quad I_Q(gx)=e^{-2\tr R_a}I_Q(x).
\end{align*}
Hence we have
\begin{align*}
\Delta(gx)=e^{-2\tr T_a}e^{-2\tr R_a}\Delta(x),
\end{align*}
and
\begin{equation}\label{eq:252}
\log\Delta(e^{\lambda T_a}e)=\log\Delta(e)-2\lambda(\tr T_a+\tr  R_a)
\end{equation}
for $\lambda\in\mathbb{R}$.
From the Taylor expansion of the left hand side of \eqref{eq:252}, we see that
\begin{align*}
\partial_a\log\Delta(e)&=-2(\tr T_a+\tr R_a),\\
\partial_{a^2}\log\Delta(e)+\partial_a^2\log\Delta(e)&=0,
\end{align*}
which leads to the desired expression.
\end{proof}
\begin{remark}\label{rem:e'x}
By Lemmas \ref{lem:5.5} and \ref{lem:253}, we have
\begin{align*}
\langle e',x\rangle_U=2(\tr T_x+\tr R_x)\quad(x\in U).
\end{align*}
\end{remark}
\begin{proposition}\label{prop:comm}
One has $[A,T_{e'}]=0\,(A\in\mathfrak{g}(\Omega))$.
\end{proposition}
\begin{proof}
For $x,y\in U$, We have 
\begin{align*}
\langle AT_{e'}x,y\rangle_U=\langle T_{e'}x,{}^tAy\rangle_U=\langle T_xe',{}^tAy\rangle_U=\langle e',x({}^tAy)\rangle_U.
\end{align*}
By Remarks \ref{rem:254}, \ref{rem:e'x}, and \eqref{eq:raxy}, this equals
\begin{align*}
2(\tr T_{x({}^tAy)}+\tr R_{x({}^tAy)})
&=2(\tr T_{(Ax)y}+\tr R_{(Ax)y})
\\&=\langle e',(Ax)y\rangle_U
\\&=\langle T_{Ax}e',y\rangle_U=\langle T_{e'}Ax,y\rangle_U,
\end{align*}
which completes the proof.
\end{proof}

For $a\in\mathfrak{g}^V$, let $a^\#$ be the vector field on ${\mathcal{S}(\Omega,Q)}$ defined by
\begin{align*}
{a^\#}_{\bf z}:=\frac{d}{dt}\Bigr|_{t=0}e^{ta}{\bf z}\quad({\bf z}\in \mathcal{S}(\Omega,Q)).
\end{align*}
Let us denote the complex structure of ${\mathcal{S}(\Omega,Q)}$ by $J\in T{\mathcal{S}(\Omega,Q)}\otimes T^*{\mathcal{S}(\Omega,Q)}$.
For $c=(iy,js)\in C:=(i\Omega\times {{{{{{{{{{jS}}}}}}}}}})\cap{\mathcal{S}(\Omega,Q)}$, let
\begin{align*}
{\bf n}(W):=\{{\bf n}(0,w)\mid w\in W\},\quad H_c:=\{{\bf n}(x,R_{y-Q(s,s)}s')\mid x\in U,s'\in jW^{\perp,g_{e'}}\}.
\end{align*}
Then we have the following proposition.
\begin{proposition}\label{prop:5.13}
One has $(T_cG^Wc)^{\perp,\,\tilde{g}_c}=jT_cH_cc$.
\end{proposition}
\begin{proof}
Noting \eqref{eq:5.1}, and comparing the dimensions, we only need to prove that
\begin{align*}
(T_c{\bf n}(W)c)^{\perp,\,\tilde{g}_c}\supset jT_cH_cc.
\end{align*}
Taking $A\in\mathfrak{g}(\Omega)$ such that $e^Ae=y-Q(s,s)$, we have for $w\in W$,
\begin{align*}
&\tilde{g}(x^\#+(R_{y-Q(s,s)}s')^\#,Jw^\#)(c)
\\&=\tilde{g}_{(ie,0)}(e^{-A}x-e^{-A}[js,R_{e^Ae}s']+e^{-\beta(A)}R_{e^Ae}s',je^{-A}[js,w]+je^{-\beta(A)}w)
\\&=\tilde{g}_{(ie,0)}(e^{-\beta(A)}R_{e^Ae}s',je^{-\beta(A)}w)
\\&=\tilde{g}_{(ie,0)}(R_ee^{\beta(A)^*}s',je^{-\beta(A)}w)\\
\\&=\frac{1}{2}\tilde{g}_{(ie,0)}(e^{\beta(A)^*}s',je^{-\beta(A)}w)
\\&=\frac{1}{2}\langle e',\re Q(e^{\beta(A)^*}s',je^{-\beta(A)}w)\rangle_U
\\&=\frac{1}{2}\langle e',\re Q(s',jw)\rangle_U=0.
\end{align*}
Here the first, second, third, forth, fifth, and sixth equalities follow from the invariance of the Bergman metric, the formulas \eqref{eq:5.1} and \eqref{eq:5.2}, Lemma \ref{lem:105}, \eqref{eq:214}, \eqref{eq:gxgtilde}, and Lemma \ref{prop:comm}.
This completes the proof.
\end{proof}
We shall prove \ref{condi:mainreal} $\Rightarrow$ \ref{condi:maincoiso} of Theorem \ref{th:main}.
\begin{theorem}\label{th:5.14}
Suppose that $\im Q(S,S)=\{0\}$.
Then every $G^W$-orbit of ${\mathcal{S}(\Omega,Q)}$ is a coisotropic submanifold with respect to the Bergman metric.
\end{theorem}
\begin{proof}
We have $G^WC={\mathcal{S}(\Omega,Q)}$.
Thus it is enough to show that for each $c\in C$,
\begin{align*}
(T_cG^Wc)^{\perp,\,\tilde{g}_c}\subset j T_c G^Wc.
\end{align*}
In addition, we only need to show that $H_c\subset G^W\,(c\in C)$ according to Proposition \ref{prop:5.13}.
For this, we can see that for $x\in\ U$ and $s,s'\in S$,
\begin{align*}
0&=-\langle xe'^{-1},\im Q(s,s')\rangle_U=2\re h(jR_{xe'^{-1}}s,s').
\end{align*}
On the other hand, we have
\begin{align*}
jR_{xe'^{-1}}S=(R_x(R_{e'})^{-1}+(R_{e'})^{-1}R_{x})R_{e'}W^{\perp,\,g_{e'}}=R_xW^{\perp,\,g_{e'}},
\end{align*}
where the second equality follows from Lemma \ref{prop:comm}.
Then we get the desired relation $jR_xW^{\perp,\,g_{e'}}\subset W\,(x\in U)$.
\end{proof}
\begin{remark}\label{rem:256}
We see from \eqref{eq:5.1} that for $a,a',b,b'\in U$,
\begin{align*}
\tilde{g}_{(ie,0)}(a+ib,a'+ib')=(\partial_{\frac{a}{2}}\partial_{\frac{a'}{2}}+\partial_{\frac{b}{2}}\partial_{\frac{b'}{2}})\log\Delta(e),
\end{align*}
which equals
\begin{align*}
-\frac{1}{2}(\tr T_{aa'+bb'}+\tr R_{aa'+bb'})
\end{align*}
by Lemma \ref{lem:253}.
This fact together with Remark \ref{rem:254} and \eqref{eq:raxy} shows that the adjoints of $A\in\mathfrak{g}(\Omega)$ with respect to $\langle\cdot,\cdot\rangle_U$ and $\tilde{g}_{(ie,0)}|_{U\times U}$ coincide.
\end{remark}

\medskip


\end{document}